\documentclass[final,leqno,showlabe]{siamltex}
\usepackage[bookmarksopen,bookmarksopenlevel=0,bookmarksdepth=2]{hyperref}
\usepackage{amsmath}
\usepackage{amssymb}
\usepackage{cases}
\usepackage{enumerate}
\usepackage{graphicx}
\usepackage{indentfirst}
\usepackage{doi}
\usepackage{cleveref}
\usepackage[usenames]{color}
\usepackage{wrapfig}
\usepackage{stmaryrd}
\usepackage{array} 
\usepackage{graphicx,epstopdf}
\usepackage[caption=false]{subfig}
\usepackage{mathrsfs}
\usepackage{multirow}
\SetSymbolFont{stmry}{bold}{U}{stmry}{m}{n}

\newtheorem{thm}{Theorem}[section]
\newtheorem{prop}{Proposition}[section]
\newtheorem{remark}{Remark}[section]

\newcommand{\norm}[1]{\lVert#1\rVert}
\newcommand{\vnorm}[1]{\left\lVert#1\right\rVert}
\newcommand{\abs}[1]{\left\lvert#1\right\rvert}
\newcommand{\bmath}[1]{\mbox{\boldmath{$#1$}}}
\newcommand{\be}{\begin{equation}}
\newcommand{\ee}{\end{equation}}

\newcommand{\rd}{\mathrm{d}}
\newcommand{\dA}{\mathrm{d}A}
\newcommand{\Vol}{\mathrm{Vol}}

\newcommand{\nn}{\nonumber}
\renewcommand{\vec}[1]{\mathbf{#1}}

\title{A structure-preserving parametric finite
element method for surface diffusion}

\author{Weizhu Bao\thanks{Department of Mathematics, National University of Singapore, Singapore, 119076 (matbaowz@nus.edu.sg). This author's research was supported by the Ministry of Education of Singapore grant
MOE2019-T2-1-063 (R-146-000-296-112).}
\and Quan Zhao\thanks{Department of Mathematics, National University of Singapore, Singapore 119076 (quanzhao90@u.nus.edu). This author's research was supported by the Ministry of Education
of Singapore grant R-146-000-285-114.}
}

\date{}

\numberwithin{equation}{section}
\renewcommand{\theequation}{\arabic{section}.\arabic{equation}}

\begin{document}
\maketitle

\begin{abstract}
We propose a structure-preserving parametric finite element method (SP-PFEM) for discretizing the surface diffusion of a closed curve in two dimensions (2D) or surface in three dimensions (3D). Here the ``structure-preserving'' refers to preserving the two fundamental geometric structures of the surface diffusion flow: (i) the conservation of the area/volume enclosed by the closed curve/surface, and (ii) the decrease of the perimeter/total surface area of the curve/surface. For simplicity of notations, we begin with the surface diffusion of a closed curve in 2D and present a weak (variational) formulation of the governing equation. Then we discretize the variational formulation by using the backward Euler method in time and piecewise linear parametric finite elements in space, with a
proper approximation of the unit normal vector by using the information of the curves at the current and next time step. The constructed numerical method is shown to preserve the two geometric structures and also enjoys the good property of asymptotic equal mesh distribution. The proposed SP-PFEM is ``weakly'' implicit (or almost semi-implicit) and the nonlinear system at each time step can be solved very efficiently and accurately by the Newton's iterative method.  The SP-PFEM is then extended to discretize the surface diffusion of a closed surface in 3D. Extensive numerical results, including convergence tests, structure-preserving property and asymptotic equal mesh distribution, are reported to demonstrate the accuracy and efficiency of the proposed SP-PFEM for simulating surface diffusion in 2D and 3D.
\end{abstract}




\pagestyle{myheadings} \markboth{W.~Bao and Q.~Zhao}
{Structure-preserving PFEM for surface diffusion}

\section{Introduction}


Surface diffusion is a general process involving the motion of adatoms, molecules, and atomic clusters at solid material surfaces \cite{Oura2013surface}. It is an important transport mechanism or kinetic pathway in  epitaxial growth, surface phase formation, heterogeneous catalysis and other areas in surface sciences \cite{Shustorovich1991}. In fact,
surface diffusion has found broader and significant applications in materials science and solid-state physics, such as the crystal growth of nanomaterials \cite{Gomer1990, Gilmer1972simulation} and solid-state dewetting~\cite{Srolovitz86, Wang15, Jiang19c}.

\begin{figure}[t]
\centering
\includegraphics[width=0.75\textwidth]{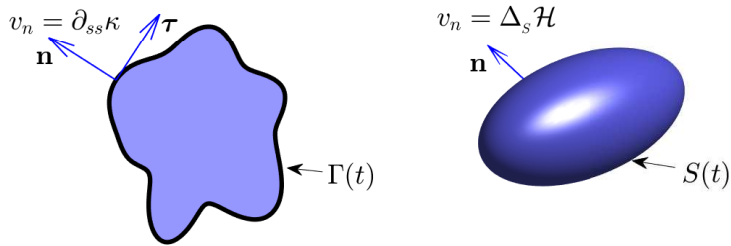}
\caption{A schematic illustration of surface diffusion of a closed curve $\Gamma(t)$ in 2D (left panel) and a closed surface $S(t)$ in 3D (right panel), where $\vec n$ is the outward unit normal vector, and $\bmath{\tau}$ represents the unit tangential vector of the curve in 2D.}
\label{fig:SurfaceDiffusion}
\end{figure}


To describe the evolution of microstructure in polycrystalline materials, Mullins firstly developed a mathematical formulation for surface diffusion  \cite{Mullins57}. As is shown in Fig.~\ref{fig:SurfaceDiffusion}, the motion by surface diffusion for a closed curve in two dimensions (2D) or a closed surface in three dimensions (3D) is governed by the following geometric evolution equations
\cite{Mullins57, Cahn94}
\begin{subequations}
\begin{numcases}{v_n=}
\label{eq:sf2d}
\partial_{ss}\kappa,&\text{in 2D},\\[0.3em]
\Delta_{_S}\mathcal{H},&\text{in 3D},
\label{eq:sf3d}
\end{numcases}
\end{subequations}
where $v_n$ is the normal velocity, $\kappa$ represents the curvature of the 2D curve with $s$ being the arc length parameter, $\mathcal{H}$ represents the mean curvature of the 3D surface with $\Delta_{_S}$ denoting the Laplace-Beltrami operator on the surface, i.e. $\Delta_{_S}:=\nabla_{_S}\cdot\nabla_{_S}$ with
$\nabla_{_S}$ denoting the surface gradient operator. It is well-known that surface diffusion has the following two essential geometric properties:
\begin{itemize}
\item [(1)]  the area of the region enclosed by the 2D curve and the volume of the region enclosed by the 3D surface are conserved;
\item [(2)] the perimeter of the 2D curve and the total surface area of the 3D surface decrease in time.
\end{itemize}
More precisely, motion by surface diffusion is the $H^{-1}$-gradient flow of the perimeter or surface area functional \cite{Mayer01}. Theoretical investigations of surface diffusion flow about the regularity and well-posedness of solutions can be found in \cite{Elliott97,Giga98,Escher98} and references therein. For numerical approximations, it is desirable to preserve the two fundamental geometric properties.


Much numerical effort has been devoted for simulating the evolution of a 2D curve or 3D surface under surface diffusion flow. Most of the early works were focused on the surface diffusion of graphs, in which the curve\slash surface is represented by a height function. In \cite{Coleman96}, a space-time finite element method for axially symmetric surfaces is developed, and the method conserves the volume and decreases the surface area. In \cite{Bansch04}, B\"ansch, Morin and Nochetto presented a weak formulation for graphs together with priori error estimates for the semi-discrete discretization. In particular, the fully discrete approximation satisfies the conservation of the enclosed volume and the decrease of the surface area. This work was later extended to the anisotropic case by Deckelnick, Dziuk and Elliott in \cite{Deckelnick05fully}. In \cite{Xu09}, Xu and Shu presented a local discontinuous Galerkin finite element method.


Recently different numerical methods have been proposed and analyzed for general curves/surfaces via different formulations and/or parametric variables. Numerical approximations in the framework of finite difference method can be found in \cite{Wong00,Mayer01,Wang15} and references therein, and the property of area/volume conservation is not considered for the corresponding discretized solution. {   Numerical approximation based on parametric formulation of surface diffusion of closed curves are considered in \cite{Dziuk2002}}. In \cite{Bansch05}, B\"ansch, Morin and Nochetto developed a finite element method for surface diffusion flow
via a complicated variational formulation and proper parametric variables. The numerical method decreases the surface area in time, but does not preserve the enclosed volume in the full discretization. In these numerical works, mesh regularisation/smoothing algorithms or artificial tangential velocities are generally required to prevent the possible mesh distortion. Based on the previous works \cite{Dziuk90,Barrett20}, Barrett, Garcke, and N\"urnberg (denoted as BGN) introduced a novel weak formulation for surface diffusion equation and presented an elegant semi-implicit parametric finite element method (PFEM) \cite{Barrett07, Barrett08JCP,Barrett19}. The PFEM is unconditionally stable by decreasing the perimeter/surface area and has the good property with respect to the mesh points distribution. Nevertheless, the fully discretized approximation fails to conserve the enclosed volume. Very recently, an area-conserving and perimeter-decreasing PFEM is proposed in \cite{Jiang2021} for a closed curve in 2D. {    In the PFEM, the unit normal and tangential vector are approximated on average in order to preserve the two geometric properties for the discretized solutions.}  However, the method is fully implicit and the mesh quality is not well preserved during time evolution. For more related works, we refer the readers to \cite{Barrett07Ani, Hausser07,Dziuk13, Bao17,Zhao19b,Zhao20, Kovacs2020} and references therein.


The main aim of this paper is to design a structure-preserving parametric finite element method (SP-PFEM) for the surface diffusion flow so that the two underlying geometric properties are well preserved in the discretized approximation. The work is based on the discretization of the weak formulation in \cite{Barrett07, Barrett08JCP}. We follow the previous works by adopting the backward Euler  method {    with an explicit treatment of the surface integrals} in time and piecewise linear elements in space, except in the numerical treatment of the unit normal vector. {    Precisely, in a similar manner to the discretization in \cite{Jiang2021}}, we approximate the unit normal vector semi-implicitly by using the information at the current and next time step. With this treatment, the obtained PFEM not only inherits the good properties of the original PFEMs by BGN in \cite{Barrett07, Barrett08JCP} such as the unconditional stability and the good mesh distribution, but also achieves the exact conservation of the area/volume in 2D/3D. The proposed method is ``weakly'' implicit (or almost semi-implicit). That is,  there is only one nonlinear term in each equation of the system, and in particular this nonlinear term is a polynomial of degree up to two and three in 2D and 3D, respectively. Thus the SP-PFEM can be solved very efficiently by the Newton's method.


The rest of the paper is organized as follows. In \cref{sec:2dcase}, we begin  with the surface diffusion flow of a closed curve in 2D, review a weak formulation, propose a SP-PFEM with detailed proof of its area conservation and perimeter dissipation, and finally  present an iterative method for solving the resulting nonlinear system. In \cref{se:Strupfem}, we extend our SP-PFEM to the surface diffusion of a closed surface in 3D.  Extensive  numerical results are reported in \cref{se:num}, and finally some conclusions are drawn in \cref{se:con}.


\section{For closed curve evolution in 2D}
\label{sec:2dcase}

In this section, we are focused on the surface diffusion flow of a closed curve in 2D (cf. Fig. \ref{fig:SurfaceDiffusion} left). We parameterize the evolution curves $\Gamma(t)$ as
\begin{align}
\vec X(\rho,~t):=(x_1(\rho,~t),~x_2(\rho,~t))^T:\mathbb{I}\times[0,T]\to\mathbb{R}^2,\nn
\end{align}
where $\mathbb{I}=\mathbb{R}\slash\mathbb{Z}=[0,~1]$ is the periodic unit interval. The arc length parameter $s$ is then computed by $s(\rho,t)=\int_0^\rho \abs{\partial_q\vec{X}}\;\rd q$ with $\partial_\rho s=\abs{\partial_\rho\vec{X}}$. We  then rewrite \eqref{eq:sf2d} into the following coupled second-order nonlinear geometric partial differential equations (PDEs)
\begin{subequations}
\label{eq:2dsf}
\begin{numcases}{}
\label{eq:2dsf1}
\vec n\cdot\partial_t\vec X = \partial_{ss}\kappa,\\[0.3em]
\kappa\,\vec n = -\partial_{ss}\vec X,
\label{eq:2dsf2}
\end{numcases}
\end{subequations}
where $\vec n:= -(\partial_s\vec X)^\perp$ is the outward unit normal vector with $(\cdot)^\perp$ being the clockwise rotation by $\frac{\pi}{2}$, i.e., $[(a,~b)^T]^{\perp}=(-b,~a)^T$. {    We recall that surface diffusion in 2D is the $H^{-1}$ gradient flow of the perimeter of the 2D curve, and has two essential geometric structures, i.e. area conservation and perimeter dissipation. Specifically, let $A(t)$ be the area of the enclosed region by $\Gamma(t)$ and $L(t)$ be the perimeter, then the two geometric structures for the dynamic system imply}
\begin{subequations}
\begin{align}
\label{eq:arealaw}
&\frac{\rd }{\rd t}A(t)=\int_{\Gamma(t)}(\partial_t\vec X\cdot\vec n)\,\rd s = \int_{\Gamma(t)}\partial_{ss}\kappa\,\rd s \equiv 0,\quad t\geq 0,\\
&\frac{\rd }{\rd t}L(t)=\int_{\Gamma(t)}\,(\partial_t\vec X\cdot\vec n)\,\kappa\,\rd s = -\int_{\Gamma(t)}\abs{\partial_s\kappa}^2\,\rd s\leq 0,\quad t\geq 0.
\label{eq:perimeterlaw}
\end{align}
\end{subequations}
%


\subsection{The weak formulation}
\label{sec:wf2d}

To obtain the weak formulation, we define the function space with respect to $\Gamma(t)$ as
\be
L^2(\mathbb{I}):=\Bigl\{u: \mathbb{I}\rightarrow \mathbb{R}, \;\text{and} \int_{\Gamma(t)}|u(s)|^2 \rd s
=\int_{\mathbb{I}} |u(s(\rho,~t))|^2 \partial_\rho s\, \rd\rho <+\infty \Bigr\},
\ee
equipped with the $L^2$-inner product
\be
\big(u,v\big)_{\Gamma(t)}:=\int_{\Gamma(t)}u(s)v(s)\,\rd s=
\int_{\mathbb{I}}u(s(\rho,t))v(s(\rho,t)) \partial_\rho s\,\rd\rho,
\ee
for any scalar (or vector-valued) functions $u,v\in L^2(\mathbb{I})$. We define the Sobolev spaces

\begin{align}
{    H^1(\mathbb{I}):=\left\{u: \mathbb{I}\rightarrow \mathbb{R}, \;{\rm and}\;u\in L^2(\mathbb{I}),\; \partial_{\rho} u\in L^2(\mathbb{I})\right\}.}
\end{align}

The weak formulation of Eq. \eqref{eq:2dsf} can be stated as follows \cite{Barrett07}: Given the initial curve $\Gamma(0)=\vec X(\mathbb{I},~0)$, for $t>0$, we find the evolution curves $\Gamma(t)=\vec X(\cdot,~t)\in [H^1(\mathbb{I})]^2$ and the curvature $\kappa(\cdot,~t)\in H^1(\mathbb{I})$ such that
\begin{subequations}
\label{eqn:2dweak}
\begin{align}
\label{eqn:2dweak1}
&\Bigl(\vec n\cdot\partial_t\vec X,~\psi\Bigr)_{\Gamma(t)}+\Bigl(\partial_s\kappa,~\partial_s\psi\Bigr)_{\Gamma(t)}=0,\quad\forall \psi\in H^1(\mathbb{I}),\\[0.5em]
\label{eqn:2dweak2}
&\Bigl(\kappa,~\vec n\cdot\boldsymbol{\omega}\Bigr)_{\Gamma(t)}-\Bigl(\partial_s\vec X,~\partial_s\boldsymbol{\omega}\Bigr)_{\Gamma(t)} = 0,\quad\forall\boldsymbol{\omega}\in[H^1(\mathbb{I})]^2.
\end{align}
\end{subequations}
Note here Eq.~\eqref{eqn:2dweak1} is obtained by taking inner product of \eqref{eq:2dsf1} with a test function
$\psi$, and applying the integration by parts. Similar approach to \eqref{eq:2dsf2} with a vector test function $\bmath{\omega}$, we obtain \eqref{eqn:2dweak2}.


\subsection{The discretization}
\label{sec:pfem2d}

Take $\tau>0$ as the uniform time step size and denote the discrete time levels as $t_m=m\tau$ for $m\geq 0$. Let $N\ge3$ be a positive integer and denote $h = 1/N$. Then a uniform partition of the reference domain $\mathbb{I}$ is given by $\mathbb{I}=\bigcup_{j=1}^N\mathbb{I}_j$, where $\mathbb{I}_j = [\rho_{j-1},~\rho_j]$ for $j=1,2,\ldots,N$ with $\rho_j = j\,h$ for $j=0,1,\ldots,N$.  Define the finite element space as
\begin{equation*}
V^h(\mathbb{I}):=\left\{u\in C(\mathbb{I}):\; u\mid_{\mathbb{I}_{j}}\in \mathcal{P}^1(\mathbb{I}_j),\; \forall \, j=1,2,\ldots,N\right\}\subseteq H^1(\mathbb{I}),
\end{equation*}
where $\mathcal{P}^1(\mathbb{I}_j)$ denotes the space of polynomials with degree at most $1$ over the subinterval $\mathbb{I}_j$. Let $\Gamma^m:=\vec X^m(\rho)=(x^m(\rho),~y^m(\rho))^T\in [V^h(\mathbb{I})]^2$ be the numerical approximation of the solution $\vec X(\cdot,t_m)$.
Then $\{\Gamma^m\}_{m\geq 0}$ are a sequence of polygonal curves consisting of connected line segments. In order to have non-degenerate meshes, we shall assume that the polygonal curves satisfy
\begin{equation}\label{hjt}
\min_{1\leq j\leq N} |\vec h_j^m|>0, \quad \hbox{with}
\quad \vec h_j^m=\vec X^m(\rho_j)-\vec X^m(\rho_{j-1}),\quad m\geq 0,
\end{equation}
where $|\vec h_j^m|$ is the length of $\vec h_j^m$ for $j=1,2,\ldots,N$.

For two piecewise continuous functions $u,v$ defined on the interval $\mathbb{I}$ with possible jumps at the nodes $\{\rho_j\}_{j=1}^{N}$, we define the mass lumped inner product $\big(\cdot,\cdot\big)_{\Gamma^m}^h$ (composite trapezoidal rule):
\begin{equation}
\label{eq:norm2d}
\big(u,~v\big)_{\Gamma^m}^h:=\frac{1}{2}\sum_{j=1}^{N}\abs{\vec h_j^m}\left[\big(u\cdot v\big)(\rho_j^-)+\big(u\cdot v\big)(\rho_{j-1}^+)\right],
\end{equation}
where $u(\rho_j^\pm)=\lim\limits_{\rho\to \rho_j^\pm} u(\rho)$ are the one-sided limits.

Let $\kappa^{m+1}\in V^h(\mathbb{I})$ be the numerical approximation of the curvature of $\Gamma^{m+1}$. We propose the full discretization
of the weak formulation in \eqref{eqn:2dweak} as follows:  Given the initial curve $\Gamma^0:=\vec X^{0}(\cdot)\in [V^h(\mathbb{I})]^2$, for $m\geq 0$, we seek the evolution curves $\Gamma^{m+1}:=\vec X^{m+1}(\cdot)\in [V^h(\mathbb{I})]^2$ and the curvature $\kappa^{m+1}(\cdot)\in V^h(\mathbb{I})$ such that the following two equations hold
\begin{subequations}
\label{eqn:dis2d}
\begin{align}
\label{eqn:dis2d1}
&\Bigl(\frac{\vec X^{m+1} - \vec X^m}{\tau}\cdot\vec n^{m+\frac{1}{2}},~\psi^h\Bigr)_{\Gamma^{m}}^h + \Bigl(\partial_s\kappa^{m+1},~\partial_s\psi^h\Bigr)_{\Gamma^m} = 0,\quad\forall\psi^h\in V^h(\mathbb{I}),\\[0.3em]
\label{eqn:dis2d2}
&\Bigl(\kappa^{m+1},~\vec n^{m+\frac{1}{2}}\cdot\boldsymbol{\omega}^h\Bigr)_{\Gamma^{m}}^h - \,\Bigl(\partial_s\vec X^{m+1},~\partial_s\boldsymbol{\omega}^h\Bigr)_{\Gamma^m}=0,
\quad\forall\boldsymbol{\omega}^h\in [V^h(\mathbb{I})]^2,
\end{align}
\end{subequations}
where $s$ is the arc length of $\Gamma^m$ and $\vec n^{m+\frac{1}{2}}$ is defined as
\begin{align}
\vec n^{m+\frac{1}{2}} = -\frac{1}{2}\left(\partial_s\vec X^m + \partial_s\vec X^{m+1}\right)^\perp=-\frac{1}{2}|\partial_{\rho}\vec X^m|^{-1}\left(\partial_{\rho}\vec X^m +\partial_{\rho}\vec X^{m+1}\right)^\perp,
\label{eq:2dseminormal}
\end{align}
and for any $f\in V^h(\mathbb{I})$, we compute its derivative with respect to the arc length parameter on $\Gamma^m$ as $\partial_sf = |\partial_{\rho}\vec X^m|^{-1}\partial_{\rho}f$.

 We will show in \cref{sec:property2d} that the  approximation of $\vec n $ using \eqref{eq:2dseminormal} contributes to the property of area conservation. The discretization is ``weakly implicit'' with only one nonlinear term introduced in \eqref{eqn:dis2d1} and \eqref{eqn:dis2d2}, respectively. In particular, the nonlinear term is a polynomial function of degree at most two with respect to the components of $\vec X^{m+1}$ and $\kappa^{m+1}$. {     We note that in \cite{Barrett07}, the unit normal is approximated explicitly by $\vec n^m:=-(\partial_s\vec X^m)^\perp$, which leads to a system of linear algebraic equations. Besides, a fully implicit PFEM for surface diffusion of closed curves is studied in \cite{Barrett2011}. However, these two methods do not preserve the enclosed area in the discretized level, e.g. the error to the area is at first-order accurate with respect to the time step size $\tau$ and is at second-order accurate with respect to the mesh size $h$.  }

 \begin{remark} The first terms in \eqref{eqn:dis2d1} and \eqref{eqn:dis2d2} are approximated using the mass lumped inner product \eqref{eq:norm2d} in order to maintain the asymptotic equal mesh distribution \cite{Barrett07, Zhao20}. Therefore, no re-meshing for the polygonal curve is needed during the time evolution.
 \end{remark}

{
\begin{remark}
In \cite{Jiang2021}, Jiang and Li proposed a new variational formulation for surface diffusion of a 2D curve and approximated the unit normal using a similar formulation in \eqref{eq:2dseminormal} so that the property of area conservation is achieved. Nevertheless, their numerical method is fully implicit and the mesh quality is not well preserved during the time evolution.
\end{remark}
}


\subsection{Area conservation and perimeter dissipation}
\label{sec:property2d}

For simplicity, denote $\vec X^m(\rho_j) = (x_j^m, ~y_j^m)^T$ for $j=0,1,\cdots, N$. We let $A^m$ be the total enclosed area and $L^m$ be the perimeter of $\Gamma^m$, then they can be written as
\begin{align}
\label{eq:disAreaLength}
A^m:=\frac{1}{2}\sum_{j=1}^N(x_j^m-x_{j-1}^m)(y_j^m+y_{j-1}^m),\qquad L^m:=\sum_{j=1}^N\abs{\vec h_j^m},\qquad m \geq 0,
\end{align}
where $\vec h_j^m$ is defined in \eqref{hjt}.

{    Similar to the work in \cite{Jiang2021}, we can prove the exact area conservation for the numerical method \eqref{eqn:dis2d}.}


\begin{thm}[Area conservation] \label{thm:2darea} Let $\Bigl(\vec X^{m+1}(\cdot),~\kappa^{m+1}(\cdot)\Bigr)$ be {    a} numerical solution of the numerical method \eqref{eqn:dis2d}. Then it holds
\begin{align}
\label{eq:areaprop}
A^{m+1} = A^m\equiv A^0,\qquad m\geq 0.
\end{align}
\end{thm}
\begin{proof}
We define the approximate solution $\Gamma^h(\alpha)=\vec X^h(\rho,~\alpha)$ via a linear interpolation of $\vec X^{m+1}$ and $\vec X^m$:
\begin{align}
\vec X^h(\rho,~\alpha) := (1-\alpha)\vec X^m(\rho) + \alpha\vec X^{m+1}(\rho),\qquad0\leq \rho\leq 1,\quad0\leq \alpha\leq 1.
\label{eq:xhde}
\end{align}
Denote by $\vec n^h$ the outward unit normal vector of $\Gamma^h(\alpha)$ and $A(\alpha)$ the area enclosed by $\Gamma^h(\alpha)$. Applying the Reynolds transport theorem to $A(\alpha)$ yields
\begin{align}
\frac{\rd }{\rd\alpha}A(\alpha)&=\int_{\Gamma^h(\alpha)}\partial_{\alpha}\vec X^h\cdot\vec n^h\,\rd s\nn\\
&=\int_{\mathbb{I}}\left[\vec X^{m+1}-\vec X^m\right]\cdot\left[-(1-\alpha)\partial_{\rho}\vec X^m - \alpha\partial_{\rho}\vec X^{m+1}\right]^\perp\,\rd\rho,
\label{eq:2dre}
\end{align}
where we revoke Eq.~\eqref{eq:xhde} and the identities:
\begin{align}
\partial_{\alpha}\vec X^h=\vec X^{m+1}-\vec X^m,\qquad
\vec n^h=-|\partial_{\rho}\vec X^h|^{-1}(\partial_{\rho}\vec X^h)^\perp.
\end{align}

Integrating \eqref{eq:2dre} with respect to $\alpha$ from $0$ to $1$ and noting \eqref{eq:2dseminormal},  we arrive at
\begin{align}
A(1) - A(0)&=\int_{\mathbb{I}}\left[\vec X^{m+1}-\vec X^m\right]\cdot\left[-\frac{1}{2}(\partial_{\rho}\vec X^m + \partial_{\rho}\vec X^{m+1})\right]^\perp\,\rd\rho\nn\\
&=\Bigl((\vec X^{m+1}-\vec X^m)\cdot\vec n^{m+\frac{1}{2}},~1\Bigr)_{\Gamma^m}^h.
\end{align}
Now setting $\psi^h=\tau$ in \eqref{eqn:dis2d1} and noting the above equation, we obtain
$A(1) = A(0)$. Thus $A^{m+1} = A^m$.
\end{proof}


Similar to the previous work in \cite{Barrett07}, we establish the unconditional stability of the numerical method \eqref{eqn:dis2d} by showing the perimeter decreases in time.
\begin{thm}[Unconditional stability]\label{thm:2dstability}Let $\Bigl(\vec X^{m+1}(\cdot),~\kappa^{m+1}(\cdot)\Bigr)$ be a numerical solution of \eqref{eqn:dis2d}. Then it holds
\begin{align}
\label{eq:energyprop}
{    L^m + \tau\sum_{l=1}^m\Bigl(\partial_s\kappa^{l},~\partial_s\kappa^{l}\Bigr)_{\Gamma^l}\leq L^0 = \sum_{j=1}^N|\vec h_j^0|,\qquad m\geq 1.}
\end{align}
\end{thm}
\begin{proof}
Setting $\psi^h = \tau\kappa^{m+1}$ in \eqref{eqn:dis2d1} and $\boldsymbol{\omega}^h=\vec X^{m+1}-\vec X^m$ in \eqref{eqn:dis2d2}, and then combining the two equations, we get
\begin{align}
\label{eq:energy1}
\tau\Bigl(\partial_s\kappa^{m+1},~\partial_s\kappa^{m+1}\Bigr)_{\Gamma^m} + \Bigl(\partial_s\vec X^{m+1},~\partial_s(\vec X^{m+1}-\vec X^m)\Bigr)_{\Gamma^m}=0.
\end{align}
Since $\partial_s$ is the derivative with respect to the arc length of $\Gamma^m$, we can compute
\begin{align}
\Bigl(\partial_s\vec X^{m+1},~\partial_s(\vec X^{m+1}-\vec X^m)\Bigr)_{\Gamma^m}
&= \sum_{j=1}^N\frac{\vec h_j^{m+1}}{|\vec h_j^m|}\cdot\left(\frac{\vec h_j^{m+1}}{|\vec h_j^m|}-\frac{\vec h_j^m}{|\vec h_j^m|}\right)|\vec h_j^m|\nn\\
&\geq \sum_{j=1}^N\frac{1}{2}\left(\abs{\frac{\vec h_j^{m+1}}{\vec h_j^m}}^2 - 1\right)|\vec h_j^m|\nn\\
&\geq \sum_{j=1}^N\left(\abs{\frac{\vec h_j^{m+1}}{\vec h_j^m}} - 1\right)|\vec h_j^m|\nn\\
& =L^{m+1} - L^{m},
\label{eq:energy2}
\end{align}
where we have used the fact $a(a - b)\geq \frac{1}{2}(a^2-b^2)$ for the first inequality and $\frac{a^2-1}{2}\geq a -1$ for the second inequality.

Plugging \eqref{eq:energy2} into \eqref{eq:energy1}, we obtain
\begin{align}
\label{eq:energy3}
L^{m+1} + \tau\Bigl(\partial_s\kappa^{m+1},~\partial_s\kappa^{m+1}\Bigr)_{\Gamma^m}\leq L^m,\qquad m\ge0.
\end{align}
{    Replacing $m$ by $l$ in Eq.~\eqref{eq:energy3}, and then summing up it for $l$ from $0$ to $m-1$ gives \eqref{eq:energyprop}.}
\end{proof}


Define the mesh ratio indicator (MRI) $\Psi^m$ of the polygonal curve $\Gamma^m$ as
\begin{equation}
\Psi^m=\frac{\max_{1\le j\le N}\;\abs{\vec h_j^m}}{\min_{1\le j\le N}\;\abs{\vec h_j^m}},
\qquad m\ge 0.
\label{eq:MRI}
\end{equation}
We then have
\begin{prop}[Asymptotic equal mesh distribution]
\label{prop:full1}
Let $\Bigl(\vec X^m(\cdot),~\kappa^m(\cdot)\Bigr)$ be a solution of the numerical method \eqref{eqn:dis2d}, and when $m\to+\infty$,
$\vec X^m(\cdot)$ and $\kappa^m(\cdot)$ converge to the equilibrium $\Gamma^e=\vec X^e(\rho)\in [V^h(\mathbb{I})]^2$ and $\kappa^e(\rho)\in V^h(\mathbb{I})$, respectively, satisfying $\min_{1\le j\le N}\;|\vec h_{j}^e|>0$ with $\vec h_{j}^e:=\vec X^e(\rho_j) - \vec X^e(\rho_{j-1})$ for $1\le j\le N$. Then we have
\begin{subequations}
\begin{align}
\label{eq:st1}
&\kappa^e(\rho)\equiv\kappa^c,\quad 0\le \rho\le1,\\
\label{eq:st2}
&\lim_{m\to +\infty} \Psi^m=\Psi^e:=\frac{\max_{1\le j\le N}\;|\vec h_{j}^e|}{\min_{1\le j\le N}\;|\vec h_{j}^e|}=1,\\
&\lim\limits_{m\to +\infty} L^m=L^e= 2\sqrt{A^0\pi}\left( 1 + \frac{\pi^2}{6}h^2 + O(h^4)\right).
\label{eq:st3}
\end{align}
\end{subequations}
\end{prop}
\begin{proof}
Eqs.~\eqref{eq:st1} and \eqref{eq:st2} follow directly from the Proposition 3.3 in \cite{Zhao20}, thus we omit the proof here. This implies that the equilibrium shape is a regular $N$-sided polygon. By noting the area conservation in Theorem \ref{thm:2darea}, we can derive
\begin{align}
L^e=2\sqrt{A^0 N \tan(\frac{\pi}{N})},\qquad h = 1/N,
\end{align}
which gives Eq.~\eqref{eq:st3} with a simple application of the Taylor expansion.
\end{proof}


\subsection{The iterative solver}
\label{sec:iterative2D}

For the resulting nonlinear system in \eqref{eqn:dis2d}, we use the Newton's iterative method for computing $\left(\vec X^{m+1},~\kappa^{m+1}\right)$. In the $i$-th iteration, given $\Bigl(\vec X^{m+1,i}(\cdot),~\kappa^{m+1,i}(\cdot)\Bigr)\in\left([V^h(\mathbb{I})]^2,~V^h(\mathbb{I})\right)$, we compute the Newton direction $\Bigl(\vec X^\delta(\cdot),~\kappa^\delta(\cdot)\Bigr)\in\left([V^h(\mathbb{I})]^2,~V^h(\mathbb{I})\right)$ such that the following two equations hold
\begin{subequations}
\label{eq:newton}
\begin{align}
\Bigl(\frac{\vec X^\delta}{\tau}\cdot\vec n^{m+\frac{1}{2},i},~\psi^h\Bigr)_{\Gamma^m}^h + \left(\frac{\vec X^{m+1,i}-\vec X^m}{\tau}\cdot\frac{(-\partial_{\rho}\vec X^{\delta})^\perp}{2\,|\partial_{\rho}\vec X^m|},~\psi^h\right)_{\Gamma^m}^h+\Bigl(\partial_s\kappa^{\delta},~\partial_s\psi^h\Bigr)_{\Gamma^m}\nn\\=-\Bigl(\frac{\vec X^{m+1,i} - \vec X^m}{\tau}\cdot\vec n^{m+\frac{1}{2},i},~\psi^h\Bigr)_{\Gamma^m}^h -\Bigl(\partial_s\kappa^{m+1,i},~\partial_s\psi^h\Bigr)_{\Gamma^m},\label{eq:newton1}\hspace{1cm}\\[0.6em]
 \Bigl(\kappa^{\delta},~\vec n^{m+\frac{1}{2},i}\cdot\bmath{\omega}^h\Bigr)_{\Gamma^m}^h + \Bigl(\kappa^{m+1,i},~\frac{(-\partial_{\rho}\vec X^{\delta})^\perp}{2\,|\partial_{\rho}\vec X^m|}\cdot\bmath{\omega}^h\Bigr)_{\Gamma^m}^h- \,\Bigl(\partial_s\vec X^{\delta},~\partial_s\bmath{\omega}^h\Bigr)_{\Gamma^m}\hspace{1.2cm}\nn\\
 = -\Bigl(\kappa^{m+1,i},~\vec n^{m+\frac{1}{2},i}\cdot\bmath{\omega}^h\Bigl)_{\Gamma^m}^h +\, \Bigl(\partial_s\vec X^{m+1,i},~\partial_s\bmath{\omega}^h\Bigr)_{\Gamma^m},\hspace{2.1cm}
 \label{eq:newton2}
\end{align}
\end{subequations}
for any $\Bigl(\bmath{\omega}^h,~\psi^h\Bigr)\in \Bigl([V^h(\mathbb{I})]^2, ~V^h(\mathbb{I})\Bigr)$, where $\vec n^{m+\frac{1}{2},i}$ is defined as
\begin{align}
\vec n^{m+\frac{1}{2},i}:=-\frac{1}{2}|\partial_{\rho}\vec X^m|^{-1}\left(\partial_{\rho}\vec X^m + \partial_{\rho}\vec X^{m+1, i}\right)^\perp.\nn
\end{align}
We then set
\begin{align}
\label{eq:2dupdate}
\vec X^{m+1,i+1} = \vec X^{m+1,i}+\vec X^{\delta},\qquad\kappa^{m+1,i+1}= \kappa^{m+1,i} + \kappa^{\delta}.
\end{align}

For each $m\geq0$, we typically choose the initial guess    $\vec X^{m+1,0}=\vec X^m,~\kappa^{m+1,0}=\kappa^m$, and then repeat the iteration (\eqref{eq:newton} and \eqref{eq:2dupdate}) until the following two conditions hold
\begin{align*}
&\vnorm{\vec X^{m+1,i+1}-\vec X^{m+1,i}}_{\infty}=\max_{1\leq j\leq N}\abs{\vec X^{m+1,i+1}(\rho_j) - \vec X^{m+1,i}(\rho_j)}\leq{\rm tol},\\
&\vnorm{\kappa^{m+1,i+1}-\kappa^{m+1,i}}_{\infty}=\max_{1\leq j\leq N}\abs{\kappa^{m+1,i+1}(\rho_j) - \kappa^{m+1,i}(\rho_j)}\leq{\rm tol},
\end{align*}
where ${\rm tol}$ is the chosen tolerance.
\begin{remark}
As discussed in \cite{Jiang2021}, Eq.~\eqref{eq:newton} is obtained by using the first-order Taylor expansion of the nonlinear system \eqref{eqn:dis2d} at the point $\left(\vec X^{m+1,i},~\kappa^{m+1,i}\right)$, and then setting $\vec X^{\delta} = \vec X^{m+1} - \vec X^{m+1,i}$, $\kappa^{\delta} = \kappa^{m+1} - \kappa^{m+1,i}$.\end{remark}

\begin{remark}One may consider the Picard iteration method as an alternative solver. In the $i$-th iteration, we find $\Bigl(\vec X^{m+1,i+1}(\cdot),~\kappa^{m+1, i+1}(\cdot)\Bigr)\in \left([V^h(\mathbb{I})]^2,~V^h(\mathbb{I})\right)$ so that the following two equations hold
\begin{subequations}
\label{eqn:picarddis2d}
\begin{align}
\label{eqn:picarddis2d1}
&\Bigl(\frac{\vec X^{m+1,i+1} - \vec X^m}{\tau}\cdot\vec n^{m+\frac{1}{2}, i},~\psi^h\Bigr)_{\Gamma^{m}}^h +\Bigl(\partial_s\kappa^{m+1,i+1},~\partial_s\psi^h\Bigr)_{\Gamma^m} = 0,\\[0.3em]
\label{eqn:picarddis2d2}
&\Bigl(\kappa^{m+1,i+1},~\vec n^{m+\frac{1}{2}, i}\cdot\boldsymbol{\omega}^h\Bigr)_{\Gamma^{m}}^h - \,\Bigl(\partial_s\vec X^{m+1,i+1},~\partial_s\boldsymbol{\omega}^h\Bigr)_{\Gamma^m}=0,
\end{align}
\end{subequations}
for any pair element $\left(\boldsymbol{\omega}^h,~\psi^h\right)\in\left([V^h(\mathbb{I})]^2, V^h(\mathbb{I})\right)$.
Similar to the previous work in \cite{Barrett07}, it is easy to show that the linear system \eqref{eqn:picarddis2d} admits a unique solution under some weak assumptions on $\vec n^{m+\frac{1}{2},i}$. Note that the Picard iteration method does not require an initial guess of the curvature during the iterations.
\end{remark}


\section{For closed surface evolution in 3D}
\label{se:Strupfem}
In this section, we are devoted to the surface diffusion of a closed surface in 3D (cf. Fig. \ref{fig:SurfaceDiffusion} right). We consider the evolving closed surface $S(t)$ with a mapping given by
\begin{align}
\vec X(\bmath{\rho},~t)=(x_1(\bmath{\rho},t),~x_2(\bmath{\rho},t),~x_3(\bmath{\rho},~t))^T: S^0\times [0,~T]\to \mathbb{R}^3,\nn
\end{align}
where $S^0:=S(0)$ is the initial surface. Then the velocity of $S(t)$ at point $\vec X$ is
\begin{align}
\label{eq:velocity}
\bmath{v}(\vec X,~t) = \partial_t\vec X(\cdot,t),\qquad \forall\;\vec X\in S(t).
\end{align}
Similar to the 2 D case, we can rewrite \eqref{eq:sf3d} into the following coupled second-order nonlinear geometric PDEs
\begin{subequations}
\label{eq:3dsf}
\begin{numcases}{}
\label{eq:3dsf1}
\vec n\cdot \partial_t\vec X = \Delta{_S}\mathcal{H},\\[0.3em]
\mathcal{H}\,\vec n = -\Delta_{_S}\vec X,
\label{eq:3dsf2}
\end{numcases}
\end{subequations}
 {    We recall that surface diffusion in 3D is the $H^{-1}$ gradient flow of the total surface area and has two essential geometric structures, i.e. volume conservation and surface area dissipation. Specifically, let $V(t)$ denote the volume of the enclosed region by $S(t)$ and $W(t)$ denote the total surface area. Then the two geometric properties imply that}
\begin{subequations}
\begin{align}
\label{eq:masslaw}
&\frac{\rd }{\rd t}V(t)=\int_{S(t)}\vec n\cdot\partial_t\vec X\,\dA = \int_{S(t)}\Delta_{_S}\mathcal{H}\,\dA\equiv 0,\quad t\geq 0,\\
&\frac{\rd }{\rd t}W(t)= \int_{S(t)}\,(\vec n\cdot \partial_t\vec X)\,\mathcal{H}\,\dA=-\int_{S(t)}\abs{\nabla_{_S}\mathcal{H}}^2\dA\leq 0,\quad t\geq 0.
\label{eq:energylaws}
\end{align}
\end{subequations}
%


\subsection{The weak formulation}\label{sec:wf3d}

We define the function space
\begin{align*}
L^2(S(t)):=\Bigl\{\psi:S(t)\to\mathbb{R},\quad \int_{S(t)}\psi^2\,\rd A<\infty\Bigr\},
\end{align*}
equipped with the $L^2$-inner product over $S(t)$
\begin{align}
\left(u,~v\right)_{S(t)}:=\int_{S(t)}u\,v\;\rd A,\quad u,v\in L^2(S(t)),
\end{align}
and the associated $L^2$-norm $\norm{u}_{S(t)}:=\sqrt{\left(u,~u\right)_{S(t)}}$. The Sobolev space $H^1(S(t))$ can be naturally defined as
\begin{align}
H^1(S(t)):=\Bigl\{ \psi\in L^2(S(t)),\;{\rm and}\; \underline{D}_i\psi\in L^2(S(t)), i=1,2,3\Bigr\},
\end{align}
where we denote $\nabla_{_S}\psi=(\underline{D}_1\psi,~\underline{D}_2\psi,~\underline{D}_3\psi)^T$(cf. Ref. \cite{Dziuk13}). \vspace{0.5em}

Then the weak formulation for the surface diffusion \eqref{eq:3dsf} can be stated as follows~\cite{Barrett08JCP}: Given the initial surface $S(0)$, for $t>0$, we use the velocity equation \eqref{eq:velocity}, then find $\bmath{v}(\cdot,~t)\in \left[H^1(S(t))\right]^3$ and the mean curvature $\mathcal{H}(\cdot,~t)\in H^1(S(t))$ such that
\begin{subequations}
\label{eqn:weakform12}
\begin{align}
\label{eq:weakform1}
&\Bigl(\bmath{v}\cdot\vec n,~\psi\Bigr)_{S(t)} + \Bigl(\nabla_{_S}\mathcal{H},~\nabla_{_S}\psi\Bigr)_{S(t)} = 0,\quad \forall\psi\in H^1(S(t)),\\
&\Bigl(\mathcal{H},~\vec n\cdot\bmath{\omega}\Bigr)_{S(t)}-\Bigl(\nabla_{_S}\vec X,~\nabla_{_S}\bmath{\omega}\Bigr)_{S(t)} = 0,\quad \forall\bmath{\omega}\in \left[H^1(S(t))\right]^3,
\label{eq:weakform2}
\end{align}
\end{subequations}
where $\left(\nabla_{_S}\vec X,~\nabla_{_S}\bmath{\omega}\right)_{S(t)}=\sum_{l=1}^3\int_{S(t)}\nabla_{_S}\, x_l\cdot\nabla_{_S}\,\omega_l\,\rd A$ for $\bmath{\omega} = (\omega_1,~\omega_2,~\omega_3)^T$.
%


\subsection{The discretization}
\label{sec:pfem3d}

Analogous to the 2D case, we take $\tau>0$ as the uniform time step size and denote the discrete time levels as $t_m=m\,\tau$ for $m\geq 0$.  We then approximate the evolution surface $S(t_m)$ by the polygonal surface mesh $S^m$ with a collection of $K$ vertices $\left\{\vec q_k^m\right\}_{k=1}^K$ and $J$ mutually disjoint triangles. That is,
\begin{align}
S^m:=\bigcup_{j=1}^{J} \overline{\sigma_j^m},\nn
\end{align}
where we assume $\sigma_j^m$ ($j=1,2,\ldots,J$) are non-degenerate triangles
in 3D. We define the finite element space
\begin{align}
\label{eq:p1space}
&\mathbb{K}^m:=\left\{u\in C(S^m):\;\left.u\right|_{\sigma_j^m}\in \mathcal{P}^1(\sigma_j^m),\quad\forall 1\leq j\leq J\right\},
\end{align}
where $\mathcal{P}^1(\sigma_j^m)$ denotes the spaces of all polynomials with degrees at most $1$ on $\sigma_j^m$. Denote $\mathbb{X}^m:=[\mathbb{K}^m]^3$. We follow the idea in \cite{Dziuk90} and parameterize $S^{m+1}$ over $S^m$ as $S^{m+1}:=\vec X^{m+1}(\cdot)\in\mathbb{X}^m$. In particular, $\vec X^{m}(\cdot)$ is the identity function in $\mathbb{X}^m$.

We take $\sigma_j^m:=\triangle\left\{\vec q_{j_k}^m\right\}_{k=1}^3$ to indicate that $\left\{\vec q_{j_1}^m,~\vec q_{j_2}^m,~\vec q_{j_3}^m\right\}$ are the three vertices of the triangle $\sigma_j^m$ and in the anti-clockwise order on the outer surface of $S^m$. Let $\vec n^m$ denote the outward unit normal vector to $S^m$.  It is a constant vector on each triangle $\sigma_j^m$ and can be defined as
\begin{align}
\label{eq:exnormal}
\vec n^m:=\sum_{j=1}^{J}\vec n_j^m\chi_{_{\sigma_j^m}},\quad{\rm with}\quad\vec n_j^m= \frac{\bmath{\mathcal{J}}\left\{\sigma_j^m\right\}}{\abs{\bmath{\mathcal{J}}\left\{\sigma_j^m\right\}}},
\end{align}
where $\chi$ is the usual characteristic function, and $\bmath{\mathcal{J}}\left\{\sigma_j^m\right\}$ is the orientation vector of $\sigma_j^m$ given by
\begin{align}
\label{eq:orientedV}
\bmath{\mathcal{J}}\{\sigma_j^m\} = \left(\vec q_{j_2}^m - \vec q_{j_1}^m\right)\times\left(\vec q_{j_3}^m - \vec q_{j_1}^m\right).
\end{align}
To approximate the inner product $\left(\cdot,~\cdot\right)_{_{S^m}}$,
we define the mass lumped inner product
\begin{align}
 \Bigl( f,~g\Bigr)_{S^m}^h := \frac{1}{3}\sum_{j=1}^{J} \sum_{k=1}^3|\sigma_j^m|\,f\left((\vec q_{j_{_k}}^m)^-\right)\cdot g\left((\vec q_{j_{_k}}^m)^-\right),
 \label{eqn:norm3d}
 \end{align}
where $|\sigma_j^m|=\frac{1}{2}|\bmath{\mathcal{J}}\left\{\sigma_j^m\right\}|$ is the area of $\sigma_j^m$,  and $f((\vec q_{j_{_k}}^m)^-)$ denotes the one-sided limit of $f(\vec x)$ when $\vec x$ approaches towards $\vec q_{j_{_k}}^m$ from triangle $\sigma_j^m$, i.e., $f((\vec q_{j_{_k}}^m)^-)=\lim\limits_{\sigma_j^m\ni\vec x\rightarrow\vec q_{j_{_k}}^m }f(\vec x)$.

Let $\mathcal{H}^{m+1}\in\mathbb{K}^m$  denote the numerical approximation of the mean curvature of $S^{m+1}$. We propose the full discretization  of the weak formulation \eqref{eqn:weakform12} as follows: Given the polygonal surface $S^0:=\vec X^0(\cdot)\in \mathbb{X}^m$, for $m\geq 0$, find the evolution surfaces $S^{m+1}:=\vec X^{m+1}(\cdot)\in\mathbb{X}^m$ and the mean curvature $\mathcal{H}^{m+1}(\cdot)\in \mathbb{K}^m$ such that
\begin{subequations}
\label{eqn:3dpfem}
\begin{align}
\label{eqn:3dpfem1}
&\left(\frac{\vec X^{m+1}-\vec X^m}{\tau}\cdot\vec n^{m+\frac{1}{2}},~\psi^h\right)_{S^m}^h + \Bigl(\nabla_{_S}\mathcal{H}^{m+1},~\nabla_{_S}\psi^h\Bigr)_{S^m} = 0,\,\forall\psi^h\in \mathbb{K}^m,\\[0.5em]
\label{eqn:3dpfem2}
&\Bigl(\mathcal{H}^{m+1},~\vec n^{m+\frac{1}{2}}\cdot\bmath{\omega}^h\Bigr)_{S^m}^h - \Bigl(\nabla_{_S}\vec X^{m+1},~\nabla_{_S}\bmath{\omega}^h\Bigr)_{S^m} = 0,\quad\forall\bmath{\omega}^h\in\mathbb{X}^m,
\end{align}
\end{subequations}
where $\vec n^{m+\frac{1}{2}}$ is a semi-implicit approximation of $\vec n$ given by
 \begin{align}
 \label{eq:3dsemi}
\vec n^{m+\frac{1}{2}} := \sum_{j=1}^{J}\vec n^{m+\frac{1}{2}}_j\chi_{_{\sigma_j^m}},\quad \vec n^{m+\frac{1}{2}}_j =\frac{\bmath{\mathcal{J}}\left\{\sigma_j^m\right\}+4\bmath{\mathcal{J}}\{\sigma_j^{m+\frac{1}{2}}\}+\bmath{\mathcal{J}}\left\{\sigma_j^{m+1}\right\}}{6\,\abs{\bmath{\mathcal{J}}(\sigma_j^m)}},
\end{align}
with $\sigma_j^m=\triangle\left\{\vec q_{j_k}^m\right\}_{k=1}^3$, $\sigma_j^{m+\frac{1}{2}}=\triangle\Bigl\{\frac{\vec q_{j_{k}}^m + \vec q_{j_{k}}^{m+1}}{2}\Bigr\}_{k=1}^3$.

The approximation of $\vec n$ using \eqref{eq:3dsemi} leads to the conservation of the total volume, although such treatment introduces a nonlinear term in \eqref{eqn:3dpfem1} and \eqref{eqn:3dpfem2}, respectively. Specially, the nonlinear term is a third-degree polynomial function with respect to the components of $\vec X^{m+1}$ and $\mathcal{H}^{m+1}$. $\nabla_{_S}$ is the operator defined on $S^m$. That is, $\forall f\in \mathbb{K}^m$,  we can compute $\nabla_{_S}f$ on a typical triangle $\sigma=\Delta\{\vec q_k\}_{k=1}^3$ of $S^m$ as
 \begin{align}
 \label{eq:disgradient}
\left. (\nabla_{_S}f)\,\right|_{\sigma}:=f_1\,\frac{(\vec q_3-\vec q_2)\times\vec n}{|\bmath{\mathcal{J}}\left\{\sigma\right\}|} + f_2\,\frac{(\vec q_1-\vec q_3)\times\vec n}{|\bmath{\mathcal{J}}\left\{\sigma\right\}|} + f_3\,\frac{(\vec q_2-\vec q_1)\times\vec n}{|\bmath{\mathcal{J}}\left\{\sigma\right\}|},
 \end{align}
 where $\vec n = \frac{\bmath{\mathcal{J}}\left\{\sigma\right\}}{\left|\bmath{\mathcal{J}}\left\{\sigma\right\}\right|}$, {    and $f_i = f(\vec q_i)$}.

\begin{remark}\label{rk:3dgoodmesh}
 The numerical method introduces an implicit tangential velocity for the polygonal mesh points. Here we apply the trapezoidal rule for numerical integrations of the first terms in \eqref{eqn:3dpfem1} and \eqref{eqn:3dpfem2}. This helps to obtain the good property with respect to the mesh distribution \cite{Barrett08JCP}. Therefore, no re-meshing for the polygonal surface is needed during the time evolution.
\end{remark}


\subsection{Volume conservation and surface area dissipation}
\label{sec:propfem3d}
For the polygonal surface $S^m:=\vec X^m(\cdot)$, we denote $V^m$ and $W^m$ as the enclosed volume and the total surface area of $S^m$, respectively.  They can be written as
\begin{align}
\label{eq:3dVolume}
{    V^m = \frac{1}{18}\sum_{j=1}^{J}\sum_{k=1}^3\vec q_{j_k}^m\cdot\bmath{\mathcal{J}}\left\{\sigma_j^m\right\},\quad W^m:=\sum_{j=1}^{J}|\sigma_j^m|,\qquad m\geq 0,}
\end{align}
where $\sigma_j^m:=\triangle\left\{\vec q_{j_k}^m\right\}_{k=1}^3$, $|\sigma_j^m|=\frac{1}{2}|\mathbf{\mathcal{J}}\{\sigma_j^m\}|$ and $\bmath{\mathcal{J}}\left\{\sigma_j^m \right\}$ is defined by\eqref{eq:orientedV}.


We have the following theorem which mimics the geometric property in \eqref{eq:masslaw}.

\begin{thm}[Volume conservation]\label{thm:3dmassconserve}
Let $\Bigl(\vec X^{m+1}(\cdot),~\mathcal{H}^{m+1}(\cdot)\Bigr)$ be a numerical solution of the numerical method in \eqref{eqn:3dpfem}. Then it holds
\begin{align}
V^{m+1}= V^m,\qquad\forall m\geq 0,
\end{align}
where $V^m$ defined in \eqref{eq:3dVolume} represents the enclosed volume by $S^m$.
\end{thm}
\begin{proof}
We introduce an approximate solution between $S^m$ and $S^{m+1}$ via the linear interpolation:
\begin{align}
\label{eq:veczd}
\vec z(\vec p,~\alpha) = (1-\alpha)\vec X^m(\vec p) + \alpha\vec X^{m+1}(\vec p),\qquad\vec p\in S^m,\quad0\leq \alpha\leq 1.
\end{align}
This gives a sequence of polygonal surfaces $S^h(\alpha):=\bigcup_{j=1}^{J}\overline{\sigma_j^h(\alpha)}$, where $\sigma_j^h(\alpha):=\triangle\left\{\vec z(\vec q_{j_k}^m,\alpha)\right\}_{k=1}^3$. In particular, $S^h(0)=S^m$ and $S^h(1)=S^{m+1}$.

We denote by $\vec n^h(\vec z)$ the outward unit normal vector to $S^h(\alpha)$ and $\Vol(S^h(\alpha))$ the volume enclosed by $S^h(\alpha)$. Taking the derivative of $\Vol(S^h(\alpha))$ with respect to $\alpha$ and applying the Reynolds transport theorem, we have
\begin{align}
\frac{\rd}{\rd\alpha}\Vol(S^h(\alpha))&=\int_{S^h(\alpha)}\partial_{\alpha}\vec z\cdot\vec n^h(\vec z)\;\rd A,\nn\\
&=\sum_{j=1}^{J}\int_{\sigma_j^m}(\vec X^{m+1}-\vec X^m)\cdot\frac{\bmath{\mathcal{J}}\left\{\sigma_j^h(\alpha)\right\}}{\left|\bmath{\mathcal{J}}\left\{\sigma_j^h(\alpha)\right\}\right|}\frac{\left|\bmath{\mathcal{J}}\left\{\sigma_j^h(\alpha)\right\}\right|}{|\bmath{\mathcal{J}}\left\{\sigma_j^m\right\}|}\,\dA\nn\\
&=\sum_{j=1}^{J}\int_{\sigma_j^m}(\vec X^{m+1}-\vec X^m)\cdot\frac{\bmath{\mathcal{J}}\left\{\sigma_j^h(\alpha)\right\}}{|\bmath{\mathcal{J}}\left\{\sigma_j^m\right\}|}\,\dA,
\label{eqn:RTdomain}
\end{align}
where in the second equality $\frac{\abs{\bmath{\mathcal{J}}\left\{\sigma_j^h(\alpha)\right\}}}{\abs{\bmath{\mathcal{J}}\left\{\sigma_j^m\right\}}}$ serves as the Jacobian determinant, and we have used the following identities
\begin{align}
\partial_{\alpha}\vec z = \vec X^{m+1}-\vec X^m,\qquad \left.\vec n^h(\vec z)\right|_{\sigma_j^h(\alpha)} = \frac{\bmath{\mathcal{J}}\left\{\sigma_j^h(\alpha)\right\}}{\left|\bmath{\mathcal{J}}\left\{\sigma_j^h(\alpha)\right\}\right|}.
\end{align}
Integrating Eq.~\eqref{eqn:RTdomain} on both sides with respect to $\alpha$ from $0$ to $1$, we arrive at
\begin{align}
&\Vol(S^h(1)) -\Vol(S^h(0)) \nn\\
&\hspace{1cm}=\int_0^1\left(\sum_{j=1}^{J}\int_{\sigma_j^m}(\vec X^{m+1} - \vec X^m)\cdot \frac{\bmath{\mathcal{J}}\left\{\sigma_j^h(\alpha)\right\}}{|\bmath{\mathcal{J}}\left\{\sigma_j^m\right\}|}\;\dA\right)\rd\alpha \nn\\
& \hspace{1cm}= \sum_{j=1}^{J}\int_{\sigma_j^m}\frac{(\vec X^{m+1} - \vec X^m)}{|\bmath{\mathcal{J}}\{\sigma_j^m\}|}\cdot\int_0^1\bmath{\mathcal{J}}\left\{\sigma_j^h(\alpha)\right\}\;\rd\alpha\,\dA,
\label{eq:vime}
\end{align}
where we have changed the order of integration and used the fact that both $\vec X^{m+1} - \vec X^m$ and $|\bmath{\mathcal{J}}\{\sigma_j^m\}|$ are independent of $\alpha$.

By \eqref{eq:orientedV} and \eqref{eq:veczd}, we note that $\bmath{\mathcal{J}}\left\{\sigma_j^h(\alpha)\right\}$ is a quadratic function with respect to $\alpha$. Therefore applying the Simpson's rule to the integration yields
\begin{align}
\int_0^1\bmath{\mathcal{J}}\left\{\sigma_j^h(\alpha)\right\}\;\rd\alpha = \frac{1}{6}\left(\bmath{\mathcal{J}}\left\{\sigma_j^h(0)\right\}+4\bmath{\mathcal{J}}\{\sigma_j^h(\frac{1}{2})\}+\bmath{\mathcal{J}}\left\{\sigma_j^h(1)\right\}\right).
\label{eq:Jsimpson}
\end{align}
By noting the definition of $\vec n_j^{m+\frac{1}{2}}$ in  \eqref{eq:3dsemi} as well as \eqref{eq:Jsimpson},  Eq.~\eqref{eq:vime} could be recast as
\begin{align}
V^{m+1}- V^m &= \sum_{j=1}^{J}\int_{\sigma_j^m}\left(\vec X^{m+1}-\vec X^m\right)\cdot\vec n_j^{m+\frac{1}{2}}\;\dA\nn\\&=
\Bigl([\vec X^{m+1}-\vec X^m]\cdot\vec n^{m+\frac{1}{2}},~1\Bigr)_{S^m}^h,
\label{eqn:vvc}
\end{align}
where we invoke the mass lumped inner product in \eqref{eqn:norm3d}.

On the other hand, setting $\psi^h=\tau$ in \eqref{eqn:3dpfem1} yields
\begin{align}
\Bigl([\vec X^{m+1}-\vec X^m]\cdot\vec n^{m+\frac{1}{2}},~1\Bigr)_{S^m}^h=0.\nn
\end{align}
Therefore we obtain $V^{m+1}=V^m$ by noting \eqref{eqn:vvc}.
\end{proof}


Similar to the previous work in \cite{Barrett08JCP}, we can establish the unconditional stability of the numerical method \eqref{eqn:3dpfem}, which mimics the geometric property in \eqref{eq:energylaws}.

\begin{thm}[Unconditional stability]\label{thm:3dEnergystability}
Let $\Bigl(\vec X^{m+1}(\cdot),~\mathcal{H}^{m+1}(\cdot)\Bigr)$ be a numerical solution of the numerical method in \eqref{eqn:3dpfem}, then it holds
\begin{align}
\label{eqn:stability2}
{    W^m + \tau \sum_{l=1}^{m}\Bigl(\nabla_{_S}\mathcal{H}^{l},~\nabla_{_S}\mathcal{H}^{l}\Bigr)_{S^l}\leq W^0=\sum_{j=1}^{J}|\sigma_j^0|,\qquad m\geq 1.}
\end{align}
\end{thm}
\begin{proof}
Setting $\psi^h = \tau\mathcal{H}^{m+1}$ in \eqref{eqn:3dpfem1} and $\bmath{\omega}^h = \vec X^{m+1}-\vec X^m$ in \eqref{eqn:3dpfem2}, combining these two equations yields
\begin{align}
\label{eq:stability1}
\tau\Bigl(\nabla_{_S}\mathcal{H}^{m+1},~\nabla_{_S}\mathcal{H}^{m+1}\Bigr)_{S^m} + \Bigl(\nabla_{_S}\vec X^{m+1},~\nabla_{_S}(\vec X^{m+1}-\vec X^m)\Bigr)_{S^m}=0.
\end{align}

{    By $a(a-b)\geq \frac{1}{2}(a^2-b^2)$, we have $\forall A=(a_{ij}), B=(b_{ij})\in \mathbb{R}^{3\times 3}$:
\begin{align}
A:(A-B) = \sum_{i=1}^3\sum_{j=1}^3 a_{ij}(a_{ij}-b_{ij})&\geq \frac{1}{2}\sum_{i=1}^3\sum_{j=1}^3\left(a_{ij}^2 - b_{ij}^2\right)=\frac{1}{2}\left(\vnorm{A}_{\rm F}^2 - \vnorm{B}_{\rm F}^2\right)\nn
\end{align}
with $\norm{\cdot}_{\rm F}$ representing the Frobenius norm.} Then we can compute
\begin{align}
\label{eq:stability2}
\Bigl(\nabla_{_S}\vec X^{m+1},~\nabla_{_S}(\vec X^{m+1}-\vec X^m)\Bigr)_{S^m}&=\sum_{j=1}^{J}\int_{\sigma_j^m}\nabla_{_S}\vec X^{m+1}\cdot\left(\nabla_{_S}\vec X^{m+1}-\nabla_{_S}\vec X^m\right)\,\dA\nn\\
&\geq \sum_{j=1}^{J}\int_{\sigma_j^m}\frac{1}{2}\left(\vnorm{\nabla_{_S}\vec X^{m+1}}_{\rm F}^2 - \vnorm{\nabla_{_S}\vec X^m}_{\rm F}^2\right)\dA\nn\\
&\geq \sum_{j=1}^{J}\left(|\sigma_j^{m+1}| -|\sigma_j^m|\right)\nn\\
&=W^{m+1} - W^m,
\end{align}
where the last inequality is due to the fact (see Lemma 2.1 in \cite{Barrett08JCP})
\begin{align*}
\frac{1}{2}\int_{\sigma_j^m}\vnorm{\nabla_{_S}\vec X^m}_{\rm F}^2\dA=|\sigma_j^m|,\qquad \frac{1}{2}\int_{\sigma_j^m}\vnorm{\nabla_{_S}\vec X^{m+1}}_{\rm F}^2\;\dA \geq |\sigma_j^{m+1}|.
\end{align*}
Combining \eqref{eq:stability1} and \eqref{eq:stability2}, we immediately obtain
\begin{align}
\label{eqn:stability1}
W^{m+1} + \tau\Bigl(\nabla_{_S}\mathcal{H}^{m+1},~\nabla_{_S}\mathcal{H}^{m+1}\Bigr)_{S^m}\leq W^m.
\end{align}
{    Replacing $m$ by $l$ in \eqref{eqn:stability1} and summing up it for $l$ from $0$ to $m-1$ yields \eqref{eqn:stability2}.}
\end{proof}


\subsection{The iterative solver}
\label{sec:iterative3D}

In a similar manner, by using the first-order Taylor expansion of \eqref{eqn:3dpfem} at point $\left(\vec X^{m+1,i},~\mathcal{H}^{m+1,i}\right)$, we obtain the Newton's iterative method for the computation of $\left(\vec X^{m+1},~\mathcal{H}^{m+1}\right)$ as follows: Given the initial guess $\vec X^{m+1,0}(\cdot)\in \mathbb{X}^m$ and $\mathcal{H}^{m+1,0}(\cdot)\in\mathbb{K}^m$,  for $i\ge 0$, we seek the Newton direction $\Bigl(\vec X^{\delta}(\cdot), ~\mathcal{H}^{\delta}(\cdot)\Bigr)\in\Bigl(\mathbb{X}^m,~\mathbb{K}^m\Bigr)$ such that the following two equations hold
\begin{subequations}
\label{eq:newton3d}
\begin{align}
\label{eq:newton3d1}
\left(\frac{\vec X^{\delta}}{\tau}\cdot\vec n^{m+\frac{1}{2}, i},~\psi^h\right)_{S^m}^h +\left(\frac{\vec X^{m+1,i}-\vec X^m}{\tau}\cdot\bmath{\mathcal{G}}^{m+\frac{1}{2},i}_{\vec X^\delta},~\psi^h\right)_{S^m}^h
+ \Bigl(\nabla_{_S}\mathcal{H}^{\delta},~\nabla_{_S}\psi^h\Bigr)_{S^m}\nn\\= -\Bigl(\frac{\vec X^{m+1,i}-\vec X^m}{\tau}\cdot\vec n^{m+\frac{1}{2},i},~\psi^h\Bigr)_{S^m}^h - \Bigl(\nabla_{_S}\mathcal{H}^{m+1,i},~\nabla_{_S}\psi^h\Bigr)_{S^m},\hspace{1cm}\\[0.7em]
\label{eq:newton3d2}
\Bigl(\mathcal{H}^{\delta},~\vec n^{m+\frac{1}{2}, i}\cdot\bmath{\omega}^h\Bigr)_{S^m}^h + \Bigl(\mathcal{H}^{m+1,i},~\bmath{\mathcal{G}}^{m+\frac{1}{2},i}_{\vec X^{\delta}}\cdot\bmath{\omega}^h\Bigr)_{S^m}^h - \Bigl(\nabla_{_S}\vec X^{\delta},~\nabla_{_S}\bmath{\omega}^h\Bigr)_{S^m}\hspace{1.5cm}\nn\\= -\Bigl(\mathcal{H}^{m+1,i},~\vec n^{m+\frac{1}{2},i}\cdot\bmath{\omega}^h\Bigr)_{S^m}^h+\Bigl(\nabla_{_S}\vec X^{m+1,i},~\nabla_{_S}\bmath{\omega}^h\Bigr)_{S^m},\hspace{2.1cm}
\end{align}
\end{subequations}
for any $\Bigl(\bmath{\omega}^h,~\psi^h\Bigr)\in\Bigl(\mathbb{X}^m,~\mathbb{K}^m\Bigr)$, where $\vec n^{m+\frac{1}{2},i}$ and $\bmath{\mathcal{G}}^{m+\frac{1}{2},i}_{\vec X^{\delta}}$ are piecewise constant vectors over $S^m$. That is, on each triangle $\sigma_j^m$, $1\leq j\leq J$, we define them as follows:
\begin{align*}
&\left.\vec n^{m+\frac{1}{2}, i}\right|_{\sigma_j^m} =\frac{\bmath{\mathcal{J}}\left\{\sigma_j^m\right\}+4\bmath{\mathcal{J}}\{\sigma_j^{m+\frac{1}{2}, i}\}+\bmath{\mathcal{J}}\left\{\sigma_j^{m+1, i}\right\}}{6\,\abs{\bmath{\mathcal{J}}\{\sigma_j^m\}}},\\[0.6em]
&\left.\bmath{\mathcal{G}}^{m+\frac{1}{2},i}_{\vec X^\delta}\right|_{\sigma_j^m} = \frac{\vec g^{23}\times \vec X^{\delta}(\vec q_{j_1}^m) + \vec g^{31}\times\vec X^{\delta}(\vec q_{j_2}^m) + \vec g^{12}\times\vec X^{\delta}(\vec q_{j_3}^m)}{6\abs{\bmath{\mathcal{J}}\{\sigma_j^m\}}},
\end{align*}
where $\sigma_j^{m+1,i} = \triangle\Bigl\{\vec X^{m+1,i}(\vec q_{j_k}^m)\Bigr\}_{k=1}^3$, $\sigma_{j}^{m+\frac{1}{2},i}=\triangle\Bigl\{\frac{\vec q_{j_k}^m + \vec X^{m+1,i}(\vec q_{j_k}^m)}{2}\Bigr\}_{k=1}^3$ and
\begin{align}
\vec g^{lk} = 2\vec X^{m+1,i}(\vec q_{j_k}^m) + \vec X^m(\vec q_{j_k}^m) - 2\vec X^{m+1,i}(\vec q_{j_l}^m) - \vec X^m(\vec q_{j_l}^m),\quad 1\leq l,k\leq 3.\nn
\end{align}
We then update
\begin{align}
\label{eq:3dupdate}
\vec X^{m+1,i+1}=\vec X^{m+1,i}+\vec X^{\delta},\qquad\mathcal{H}^{m+1,i+1}=\mathcal{H}^{m+1,i}+\mathcal{H}^{\delta}.
\end{align}

 For each $m\geq 0$, we can choose the initial guess $\vec X^{m+1,0}=\vec X^m$, $\mathcal{H}^{m+1,0}=\mathcal{H}^m$, and then repeat the iterations in \eqref{eq:newton3d} and \eqref{eq:3dupdate} until the following conditions hold
\begin{align*}
\norm{\vec X^{m+1,i+1}-\vec X^{m+1,i}}_{\infty}=\max_{1\leq j\leq K}\left|\vec X^{m+1,i+1}(\vec q_{j}^m) - \vec X^{m+1,i}(\vec q_{j}^m)\right|\leq {\rm tol},\\
\norm{\mathcal{H}^{m+1,i+1}-\mathcal{H}^{m+1,i}}_{\infty}=\max_{1\leq j\leq K}\left|\mathcal{H}^{m+1,i+1}(\vec q_{j}^m) - \mathcal{H}^{m+1,i}(\vec q_{j}^m)\right|\leq {\rm tol}.
\end{align*}

\begin{remark}
Although it seems not easy to prove the well-posedness of the linear system \eqref{eq:newton3d}, we observe in practice the iteration method performs well with a very fast convergence provided that the computational meshes don't deteriorate. Fortunately, this is guaranteed by the good mesh property of our method, as discussed in Remark \ref{rk:3dgoodmesh}.
\end{remark}

\begin{remark}It is also possible to consider the Picard iteration for computing the resulting nonlinear system in Eq.~\eqref{eqn:3dpfem}. In the $i$-th iteration, we directly seek $\Bigl(\vec X^{m+1,i+1}(\cdot),~\mathcal{H}^{m+1,~i+1}(\cdot)\Bigr)\in\Bigl(\mathbb{X}^m,~\mathbb{K}^m\Bigr)$ such that for any pair element $\Bigl(\bmath{\omega}^h,~\psi^h\Bigr)\in\Bigl(\mathbb{X}^m,~\mathbb{K}^m\Bigr)$ it holds
\begin{subequations}
\label{eqn:picard3d}
\begin{align}
\label{eqn:picard3d1}
&\left(\frac{\vec X^{m+1, i+1}-\vec X^m}{\tau}\cdot\vec n^{m+\frac{1}{2}, i},~\psi^h\right)_{S^m}^h + \Bigl(\nabla_{_S}\mathcal{H}^{m+1,i+1},~\nabla_{_S}\psi^h\Bigr)_{S^m} = 0,\\[0.6em]
\label{eqn:picard3d2}
&\Bigl(\mathcal{H}^{m+1, i+1},~\vec n^{m+\frac{1}{2}, i}\cdot\bmath{\omega}^h\Bigr)_{S^m}^h - \Bigl(\nabla_{_S}\vec X^{m+1, i+1},~\nabla_{_S}\bmath{\omega}^h\Bigr)_{S^m} = 0.
\end{align}
\end{subequations}
Similar to the previous work in \cite{Barrett08JCP}, it is easy to show that the linear system \eqref{eqn:picard3d} admits a unique solution under some weak assumptions on $\vec n^{m+\frac{1}{2},i}$. Unlike the proposed Newton' iterative method above, the Picard iterative's method only require an initial guess $\vec X^{m+1,0}$ during the iterations.
\end{remark}


\section{Numerical results}
\label{se:num}
We present several numerical experiments, including a convergence study, to test the SP-PFEM \eqref{eqn:dis2d} for 2D in \cref{sec:nu2d}  and \eqref{eqn:3dpfem} for 3D in \cref{sec:nu3d}, respectively.

In the Newton iterations, the two linear systems in \eqref{eq:newton} and \eqref{eq:newton3d} are directly solved via the sparse LU decomposition or the GMRES with preconditioner based on the incomplete LU factorization, and the iteration tolerance is chosen as ${\rm tol}=10^{-10}$.

\subsection{For closed curves in 2D}
\label{sec:nu2d}


%
\begin{table}[tbh]
\caption{Error $e_{h,\tau}$ and the rate of convergence  at three different times by using \eqref{eqn:dis2d}. The initial shapes are chosen as a $(5.6, 0.8)$ rectangle (upper panel) and an ellipse: $\frac{x^2}{2.8^2}+\frac{y^2}{0.4^2}=1$ (lower panel). In the coarse mesh,  $h=h_0=2^{-5}$, $\tau_0=0.02$.}
\label{tb:2derror}
\begin{center}
 \begin{tabular}{@{\extracolsep{\fill}}|c|cc|cc|cc|}\hline
$(h,~\tau)$ &$e_{h,\tau}(t=0.2)$ & order &$e_{h,\tau}(t=0.5)$
& order &$e_{h,\tau}(t=2.0)$ & order  \\ \hline
$(h_0,~\tau_0)$ & 5.23E-2 & - &1.05E-1 &-& 1.12E-1 &- \\ \hline
$(\frac{h_0}{2},~\frac{\tau_0}{4})$ & 1.33E-2 & 1.97 &2.66E-2 &1.97& 2.80E-2 &2.00
\\ \hline
$(\frac{h_0}{2^2},~\frac{\tau_0}{4^2})$ & 3.16E-3 & 2.07 &6.53E-3 &2.03& 7.01E-3 &2.00 \\ \hline
$(\frac{h_0}{2^3},~\frac{\tau_0}{4^3})$ & 7.38E-4 & 2.10 &1.59E-3 &2.04 &1.75E-3  &2.00\\ \hline
 \end{tabular}\vspace{0.8em}
\begin{tabular}{@{\extracolsep{\fill}}|c|cc|cc|cc|}\hline
$(h,~\tau)$ &$e_{h,\tau}(t=0.2)$ & order &$e_{h,\tau}(t=0.5)$ & order &$e_{h,\tau}(t=2.0)$ & order  \\ \hline
$(h_0,~\tau_0)$ & 3.50E-2 & - &5.59E-2 &-& 2.12E-2 &- \\ \hline
$(\frac{h_0}{2},~\frac{\tau_0}{4})$ & 7.88E-3 & 2.15 &1.36E-2 &2.04& 5.30E-3 &2.00
\\ \hline
$(\frac{h_0}{2^2},~\frac{\tau_0}{4^2})$ & 1.78E-3 & 2.14 &3.27E-3 &2.05& 1.33E-3 &2.00 \\ \hline
$(\frac{h_0}{2^3},~\frac{\tau_0}{4^3})$ & 4.20E-4 & 2.08 &7.97E-4 &2.04 &3.32E-4  &2.00\\ \hline
 \end{tabular}
\end{center}
 \end{table}

We test the convergence rate of the numerical method in \eqref{eqn:dis2d} by carrying out simulations using different mesh sizes and time step sizes. To measure the difference between two different closed curves $\Gamma_1$ and $\Gamma_2$, we adopt the manifold distance in \cite{Zhao19b}. Let $\Omega_1$ and $\Omega_2$ be the inner regions enclosed by $\Gamma_1$ and $\Gamma_2$, respectively, then the manifold distance is given by the area of the symmetric difference region between $\Omega_1$ and $\Omega_2$ \cite{Zhao19b}:
\begin{equation}\label{MG1G2}
{\rm M}(\Gamma_1,~\Gamma_2):= |\left(\Omega_1\backslash \Omega_2 \right)\cup\left( \Omega_2\backslash\Omega_1\right)|=2|\Omega_1\cup\Omega_2|-|\Omega_1|-|\Omega_2|,
\end{equation}
where $|\Omega|$ denotes the area of $\Omega$.

We denote by $\vec X_{h,\tau}^m$ the numerical approximation of the curve $\Gamma(t_m)$ using mesh size $h$ and time step size $\tau$. {    We use the time step size $\tau = O(h^2)$ due to that the discretization is first order in temporal discretiztion and second order in spatial discretization,} and the numerical errors are computed based on the manifold distance in \eqref{MG1G2} as
\begin{align}
e_{h,\tau}(t_m):={\rm M}(\vec X_{h,\tau}^m,~\vec X_{\frac{h}{2},\frac{\tau}{4}}^m),\qquad m\geq 0.
\label{eq:2derror}
\end{align}
 Initially, two different closed curves are considered: \vspace{0.4em}
\begin{itemize}
\item ``Shape 1'': a rectangle curve with $(5.6, 0.8)$ representing its length and width.
\item ``Shape 2'': an ellipse curve given by $\frac{x^2}{2.8^2} + \frac{y^2}{0.4^2} = 1$.
\end{itemize}
\vspace{0.4em}
Numerical errors are reported in Table \ref{tb:2derror}, where we observe the order of convergence can reach about $2$ in spatial discretization.

\begin{figure}[tbh]
\centering
\includegraphics[width=0.99\textwidth]{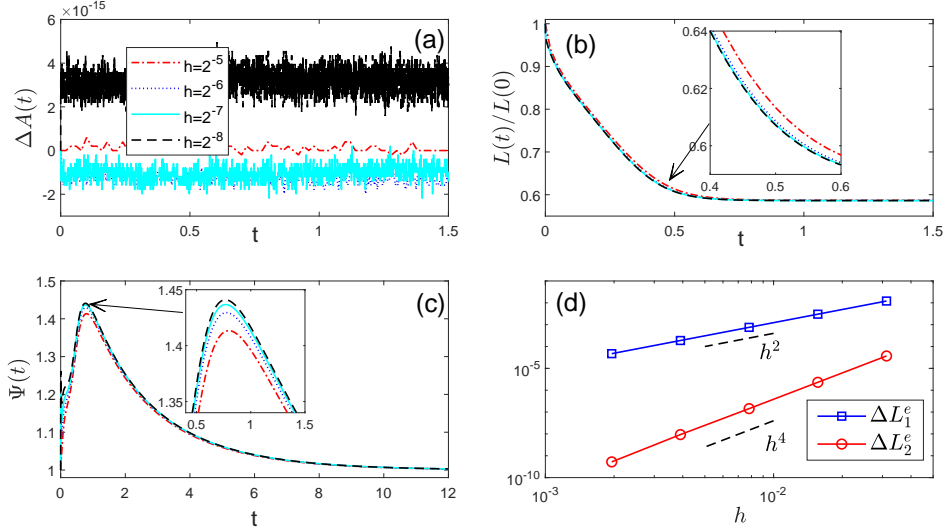}
\caption{The time history of (a) the relative area loss $\Delta A(t)$, (b) the normalized perimeter $L(t)/L(0)$ and (c) the mesh ratio indicator $\Psi(t)$ obtained by using different mesh sizes $h$ with $\tau=20.48h^2$. (d) The log-log plot of $\Delta L_1^e$ and $\Delta L_2^e$ versus the mesh size h.}
\label{fig:massenergy2d}
\end{figure}

To further assess the performance of our numerical method, we define the relative area change $\Delta A(t)$, the mesh ratio indicator $\Psi(t)$ and the perimeter errors $\Delta L^e_1$, $\Delta L^e_2$ at equilibrium:
\begin{align}
&\left.\Delta A(t)\right|_{t=t_m}:=\frac{A^m-A^0}{A^0},\qquad \left.\Psi(t)\right|_{t=t_m} = \Psi^m,\nn\\
&\Delta L^e_1:=\lim\limits_{m\to\infty}(L^m - 2\sqrt{A^0\pi}),\quad \Delta L^e_2: = \Delta L_1^e - \frac{\sqrt{A^0\pi}\pi^2}{3}h^2,\qquad m\geq 0,\nn
\end{align}
where $A^m$ and $L^m$ are given by \eqref{eq:disAreaLength},  and $\Psi^m$ is given by \eqref{eq:MRI} for the polygonal curve $\Gamma^m$.  We show the time evolution of $\Delta A(t)$ and the normalized perimeter $L(t)/L(0)$ in Fig.~\ref{fig:massenergy2d}(a),(b), respectively. It can be seen that the total area is conserved up to the machine precision under different mesh sizes, and the perimeter decreases in time. This numerically substantiates Theorem \ref{thm:2darea} and Theorem \ref{thm:2dstability}.

To examine the mesh quality during the simulations, we plot the mesh ratio indicator $\Psi(t)$ versus time in Fig.~\ref{fig:massenergy2d}(c). It is found that the mesh ratio indicator first increases to a small critical value and then gradually decreases to approximate $1$. This implies the mesh points on the polygonal curve tend to be equally distributed in the long time limit. Besides, from Fig.~\ref{fig:massenergy2d}(d), we observe that by refining the mesh size $h$, the perimeter errors $\Delta L_1^e$ and $\Delta L_2^e$ can achieve second-order and fourth-order convergence, respectively, as expected by Proposition \ref{prop:full1}.

\begin{figure}[tbh]
\centering
\includegraphics[width=0.99\textwidth]{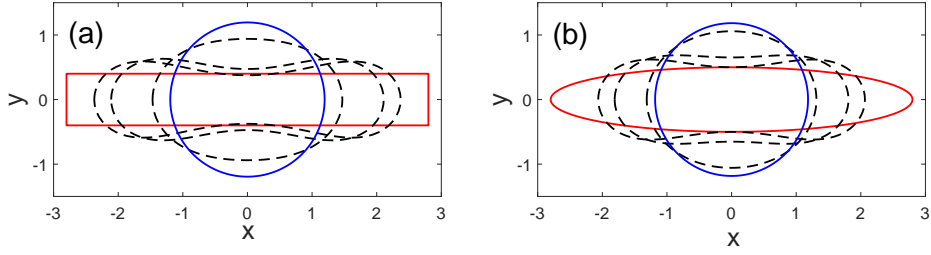}
\caption{Morphological evolutions of the closed curves towards their equilibrium shapes (blue solid line) using different initial shapes (red solid line): (a) ``Shape 1'', and (b) ``Shape 2''.}
\label{fig:Evolution}
\end{figure}

The evolutions of the curves by using the two initial shapes are depicted in Fig.~\ref{fig:Evolution}. We observe the two curves form the circle as the equilibrium shapes. We also assess the performance of the Picard iteration \eqref{eqn:picarddis2d} and the Newton's iteration \eqref{eq:newton} during the simulations. We recall that the two linear systems are solved directly with sparse LU decomposition, therefore the difference between the CPU time for each iteration is negligible. The iteration numbers for the two iterative methods are compared in Fig.~\ref{fig:PiNewC}. The Newton's method is observed to outperform the Picard iteration, since less number of iterations is needed for the former one.

\begin{figure}[!htb]
\centering
\includegraphics[width=0.8\textwidth]{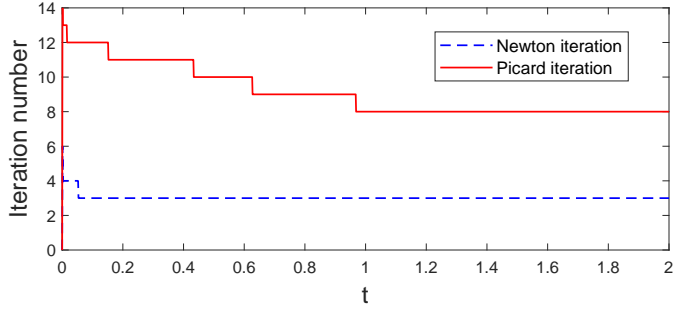}
\caption{A comparison between the number of iterations used in each time step by the Newton's method in \eqref{eq:newton} and the Picard iteration in \eqref{eqn:picarddis2d}, where we choose $h=2^{-7}, \tau=1.25\times 10^{-3}$, and ``Shape 1'' is used.}
\label{fig:PiNewC}
\end{figure}

%
\begin{figure}[!htb]
\centering
\includegraphics[width=0.95\textwidth]{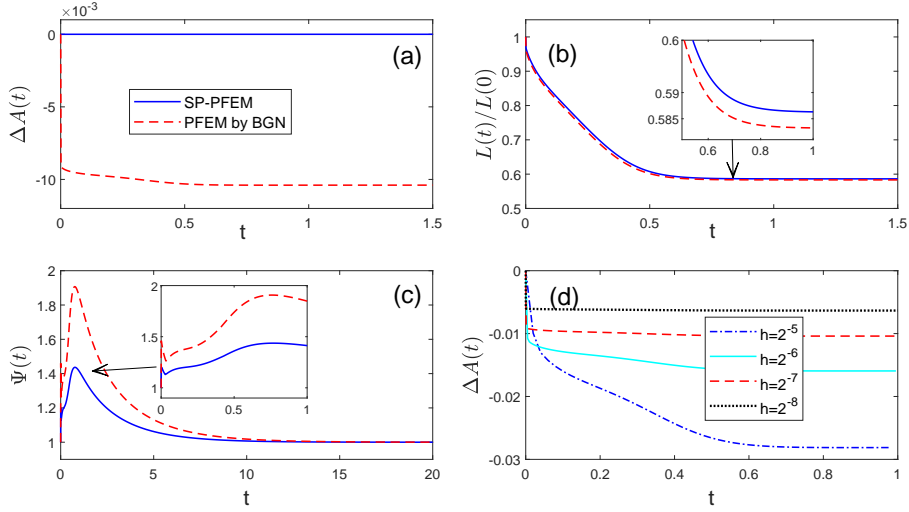}
\caption{Comparison between the SP-PFEM \eqref{eqn:dis2d} and the PFEM by BGN in \cite{Barrett07}: (a) the relative area loss  $\Delta A(t)$, (b) the normalized perimeter $L(t)/L(0)$, and (c) the mesh ratio indicator $\Psi(t)$; where we choose $h=2^{-7}, \tau=1.25\times 10^{-3}$. (d) The relative area loss $\Delta A(t)$ for the PFEM by BGN in \cite{Barrett07} with different mesh sizes $h$ and time step sizes $\tau=20.48h^2$. Here ``Shape 1'' is used. }
\label{fig:BGNC}
\end{figure}

We next conduct a comparison of our SP-PFEM and the PFEM by BGN in \cite{Barrett07}, and the numerical results are reported in Fig.~\ref{fig:BGNC}. Based on the observation, we can draw the following conclusions: (i) the time evolutions of the normalized perimeter show very good agreement between the two methods; (ii) the equal mesh distribution is achieved in the long time limit for both methods; and (iii) unlike the SP-PFEM, the PFEM by BGN in \cite{Barrett07} fails to conserve the area exactly and suffers an area loss {    up to one percent for $h=2^{-7}$, $\tau=1.25\times 10^{-3}$ or smaller, and a more detailed investigation of the area loss for the PFEM by BGN has been conducted in \cite{Barrett07,Zhao20}.}

\begin{figure}[!htb]
\centering
\includegraphics[width=0.8\textwidth]{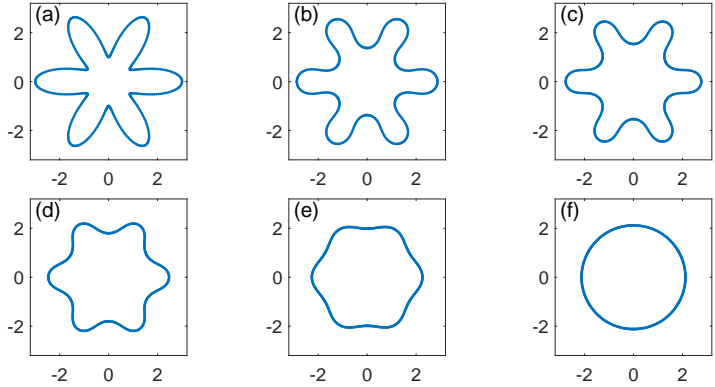}
\caption{Several snapshots in the evolution of an initially non-convex curve towards the equilibrium, where (a) $t=0$; (b) $t=0.01$; (c) $t=0.03$; (d) $t=0.06$; (e) $t=0.08$; (f) $t=0.15$.  Parameters are chosen as $h = 2^{-9}, \tau = 10^{-4}$, and the initial curve is given by Eq.~\eqref{eq:flower}. }
\label{fig:flower}
\end{figure}
\begin{figure}[!htb]
\centering
\includegraphics[width=0.95\textwidth]{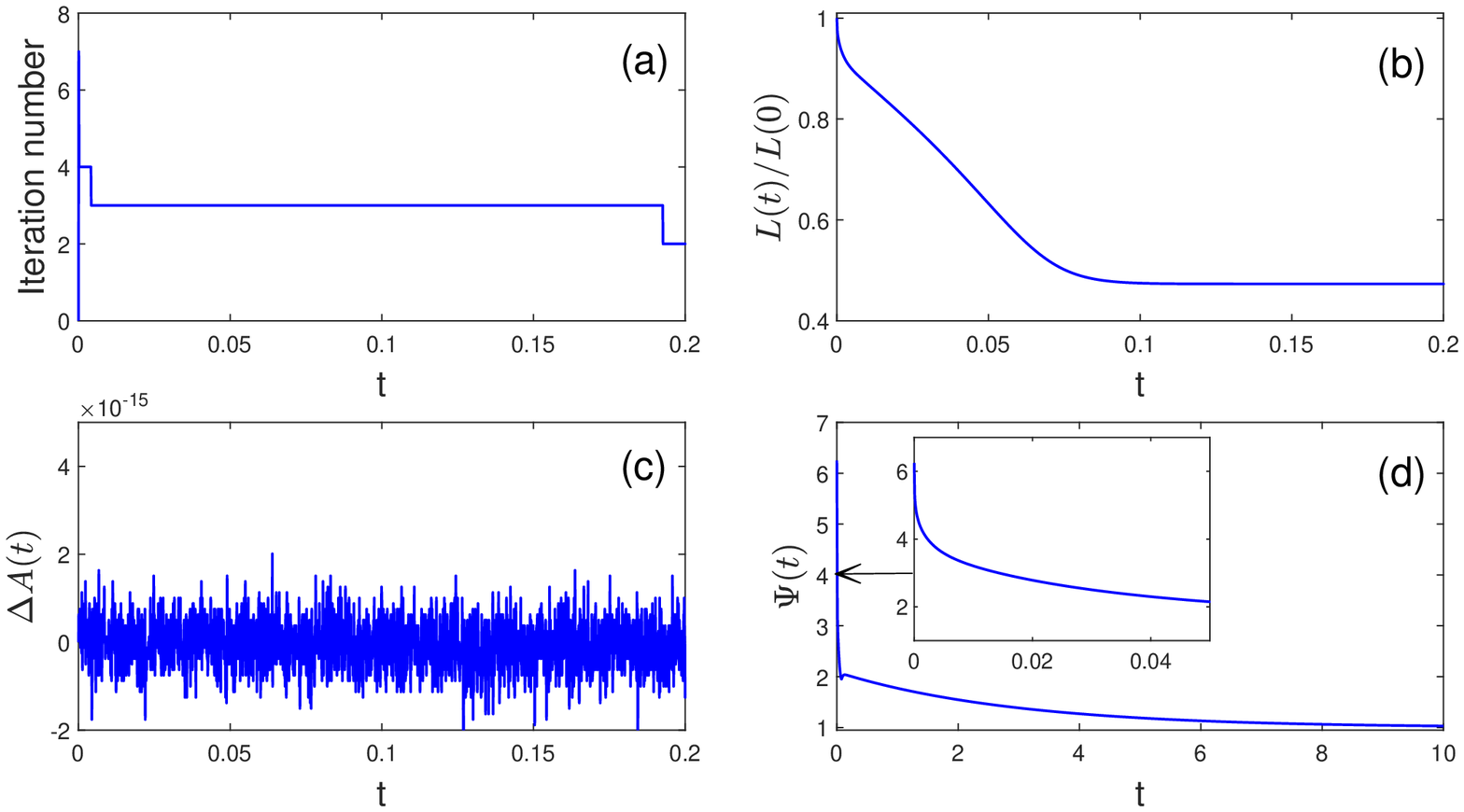}
\caption{(a) The iteration number in each time step by using the Newton's method in \eqref{eq:newton}. (b) The normalized perimeter $L(t)/L(0)$. (c) The relative area loss $\Delta A(t)$. (d) The mesh ratio indicator $\Psi(t)$. Parameters are chosen as $h = 2^{-9}, \tau = 10^{-4}$, and the initial curve is given by Eq.~\eqref{eq:flower}. }
\label{fig:flowerQ}
\end{figure}
%


We end this subsection by applying our SP-PFEM to two more complex shapes give by:

Case I. ``Shape of a flower'' with six petals:
\begin{align}
\left\{\begin{array}{l}x = [2+\cos(6\theta)]\cos\theta,\\[0.3em]
y = [2+\cos(6\theta)]\sin\theta,
\end{array}\right.\qquad\theta\in[0,~2\pi].
\label{eq:flower}
\end{align}

Case II.  ``Shape of an astroid'' with four cusps:
\begin{align}
\left\{\begin{array}{l}x =\frac{3}{4}\left[3\cos\theta + \cos(3\theta)\right],\\[0.3em]
y = \frac{3}{4}\left[3\sin\theta - \sin(3\theta)\right],
\end{array}\right.\qquad\theta\in[0,~2\pi].
\label{eq:astroid}
\end{align}

The discretization of the initial curve results from a uniform partition of the polar angle $\theta$. This yields polygonal curve with non-uniform distribution with respect to the arc length. In the simulations, we use parameters $h=2^{-9}$, $\tau=10^{-4}$. Fig.~\ref{fig:flower} depicts the curve evolution for the initial flower shape in \eqref{eq:flower}. It can be seen that the six petals gradually disappear in order to form a final circle as the equilibrium shape. The time evolution of several numerical quantities are shown in Fig.~\ref{fig:flowerQ}, where we observe the decrease of perimeter, the conservation of area as well as the long time equal mesh distribution.

\begin{figure}[!htp]
\centering
\includegraphics[width=0.8\textwidth]{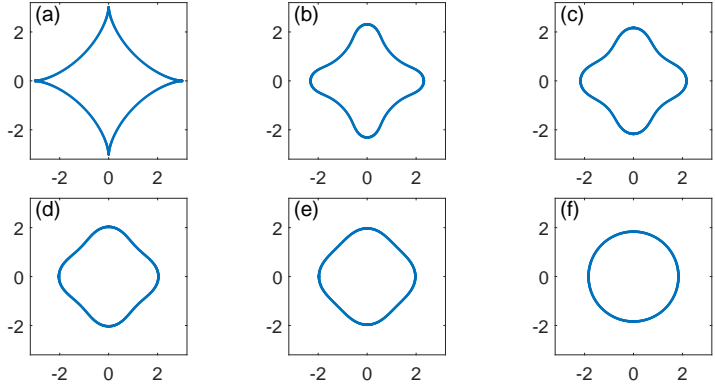}
\caption{Several snapshots in the evolution of an initially non-convex curve towards the equilibrium, where (a) $t=0$; (b) $t=0.01$; (c) $t=0.03$; (d) $t=0.06$; (e) $t=0.08$; (f) $t=0.50$. Parameters are chosen as $h = 2^{-9}, \tau = 10^{-4}$, and the initial curve is given by Eq.~\eqref{eq:astroid}. }
\label{fig:astroid}
\end{figure}

\begin{figure}[!htb]
\centering
\includegraphics[width=0.95\textwidth]{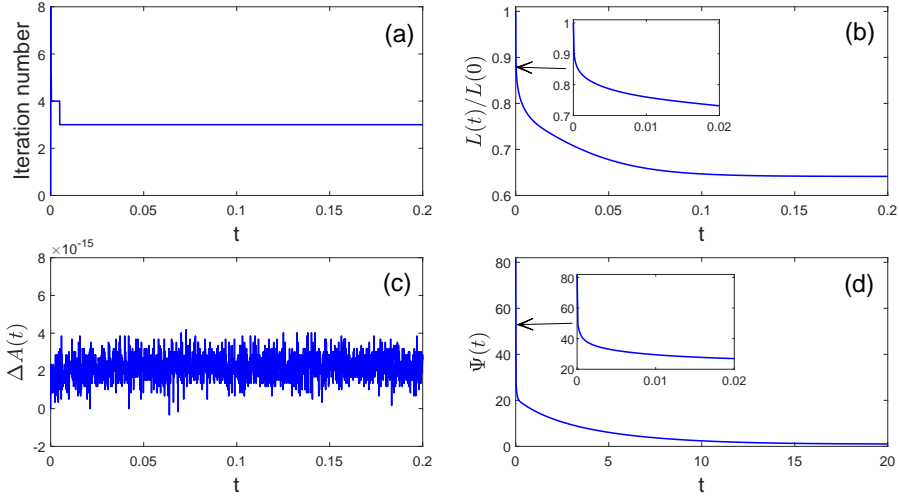}
\caption{(a) The iteration number in each time step by using the Newton's method in \eqref{eq:newton}. (b) The normalized perimeter $L(t)/L(0)$. (c) The relative area loss $\Delta A(t)$. (d) The mesh ratio indicator $\Psi(t)$.  Parameters are chosen as $h = 2^{-9}, \tau = 10^{-4}$, and the initial curve is given by Eq.~\eqref{eq:astroid}.}
\label{fig:astroidQ}
\end{figure}

Analogous numerical results for the astroid are depicted in Fig.~\ref{fig:astroid} and Fig.~\ref{fig:astroidQ}. Due to the presence of cusps, we find the mesh ratio indicator $\Psi(t)$ begins with a large value. As time evolves, the decease of $\Psi(t)$ is still observed and  the equal distribution is reached finally. These two numerical examples demonstrate the applicability and reliability of our proposed numerical method SP-PFEM.

\subsection{For closed surfaces in 3D}
\label{sec:nu3d}


\begin{table}[!htp]
\caption{Error $\widetilde{e}_{h,\tau}$ and the rate of convergence for the dynamic surface at three different times. The numerical results are obtained using \eqref{eqn:3dpfem} with initial shape given by a $(4,1,1)$ cuboid, and $h_0=0.25$, $\tau_0=0.01$.}
\label{tb:3derror}
\begin{center}
 {\rule{0.96\textwidth}{1pt}}
\begin{tabular}{@{\extracolsep{\fill}}|c|cc|cc|cc|}
$(h,\,\tau)$ &$\widetilde{e}_{h,\tau}(t=0.08)$ & order &$\widetilde{e}_{h,\tau}(t=0.2)$
& order &$\widetilde{e}_{h,\tau}(t=0.3)$ & order  \\ \hline
$(h_0,\,\tau_0)$ & 3.72E-2 & - &5.30E-2 &-& 3.91E-2 &- \\ \hline
$(\frac{h_0}{2},\, \frac{\tau_0}{4})$ & 1.06E-2 & 1.81 &1.34E-2 &1.98& 9.92E-3 &1.98
\\ \hline
$(\frac{h_0}{2^2},\,\frac{\tau_0}{4^2})$ & 2.99E-3 & 1.83 &3.53E-3 &1.92& 2.81E-3 &1.82
 \end{tabular}
  {\rule{0.96\textwidth}{0.8pt}}
\end{center}
 \end{table}

We test the convergence rate of the numerical method SP-PFEM \eqref{eqn:3dpfem} by using the example of an initial $(4,1,1)$ cuboid with $(4,1,1)$ representing its length, width, and height. We note the manifold distance in \eqref{MG1G2} can be readily extend to 3D. However, practical computations involving two polygonal surfaces can be rather complicated and tedious. Therefore, given
\begin{align*}
S:=\cup_{j=1}^J\overline{\sigma_j}\quad{\rm with\; vertices} \quad \left\{\vec q_k\right\}_{k=1}^K,\\
S^\prime=\cup_{j=1}^{J^\prime}\overline{\sigma^\prime_j}\quad{\rm  with\;vertices} \quad \left\{\vec q_k^\prime\right\}_{k=1}^{K^\prime},
\end{align*}
we consider the manifold distance in $L^{\infty}$-norm
\begin{align}
\label{eq:manifold}
\mathcal{M}\left(S,~S^\prime\right) = \frac{1}{2}\Bigl(\max_{1 \leq k \leq K^\prime}\min_{1 \leq j \leq J}\,{\rm dist\left(\vec q_k^\prime,~\sigma_j\right)} + \max_{1 \leq k \leq K}\min_{1 \leq j \leq J^\prime}\,{\rm dist\left(\vec q_k,~\sigma_j^\prime\right)}\Bigr),
\end{align}
where ${\rm dist}(\vec q,~\sigma)=\inf_{\vec p\in\sigma}\vnorm{\vec p -\vec q}$ represents the distance of the vertex $\vec q$ to the triangle $\sigma$. Analogous to
Eq.~\eqref{eq:2derror}, the numerical errors are computed by comparing $\vec X_{h,\tau}$ and $\vec X_{\frac{h}{2},\frac{\tau}{4}}$
\begin{equation}
\widetilde{e}_{h,\tau}(t=t_m):=\mathcal{M}(\vec X_{h,\tau}^m,~\vec X^m_{\frac{h}{2},\frac{\tau}{4}}),\qquad m\geq 0.
\end{equation}
In these expressions, the mesh size $h$ is defined according to the initial discretization $S^0=\cup_{j=1}^J\overline{\sigma_j^0}$ such that $h= \max_{j=1}^J\sqrt{\abs{\sigma_j^0}}$, and $\vec X_{h,\tau}^m$ represents the numerical solution of $S(t_m)$ obtained using mesh size $h$ and time step size $\tau$. In the convergence test, the numerical solutions are obtained on different meshes:
\begin{align}
(K,~J,~h)=(146,288,2^{-2}),~(578,1152,2^{-3}),~(2306,4608,2^{-4}),~(9218,18432,2^{-5}).\nn
\end{align}
Numerical errors are reported in Table \ref{tb:3derror}. It can be seen that the order of the convergence for the numerical solutions can achieve around $2$ in spatial discretization.

\begin{figure}[!htp]
\centering
\includegraphics[width=0.9\textwidth]{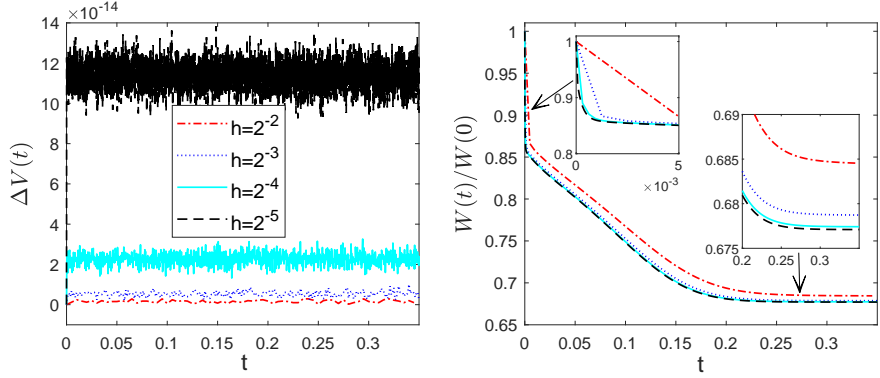}
\caption{Time history of the relative volume loss $\Delta V(t)$ (left panel) and the normalized surface area $W(t)/W(0)$ (right panel) by using different mesh sizes $h$ with $\tau = \frac{2}{25}h^2$. The initial shape is chosen as a $(4,~1,~1)$ cuboid.}
\label{fig:ME3}
\end{figure}

 \begin{figure}[!htp]
\centering
\includegraphics[width=0.95\textwidth]{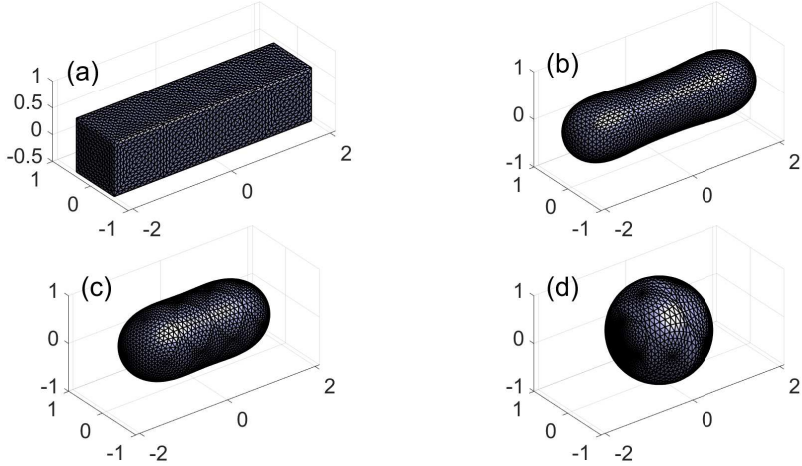}
\caption{Evolution of the polygonal mesh for an initial $(4,~1,~1)$ cuboid. (a) $t=0$; (b) $t=0.01$; (c) $t=0.1$; (d) $t=0.35$, where $h=2^{-4}$ and $\tau = 3.125\times 10^{-4}$.}
\label{fig:3dmeshE}
\end{figure}

The time evolution of the relative volume loss and the normalized surface area are depicted  in Fig.~\ref{fig:ME3}. The relative volume loss is defined as
\begin{equation}
\left.\Delta V(t)\right|_{t=t_m}:=\frac{V^m - V^0}{V^0},\qquad m\geq 0,\nn
\end{equation}
with $V^m$ given by \eqref{eq:3dVolume}. We observe the exact conservation of the volume and decrease of the surface area for the numerical solutions using different mesh sizes and time steps. Furthermore, the dynamic convergence of the normalised surface area is confirmed by refining the mesh size.


The evolution of the surface mesh with $(K,~J)=(2306,~4608)$ are shown in Fig.~\ref{fig:3dmeshE}. We observe that the sharp corners of the initial cuboid become rounded, and finally, the cuboid forms a spherical shape as the equilibrium. In particular, we observe the good mesh quality of the polygonal surface even though the re-meshing procedure is not applied. We also assess the performance of the Picard iteration \eqref{eqn:picard3d} and the Newton's iteration \eqref{eq:newton3d} during the simulations. Similar to the 2D case, the Newton method is observed to outperform the Picard iteration, as shown in Fig.~\ref{fig:3dNP}.

\begin{figure}[!htp]
\centering
\includegraphics[width=0.65\textwidth]{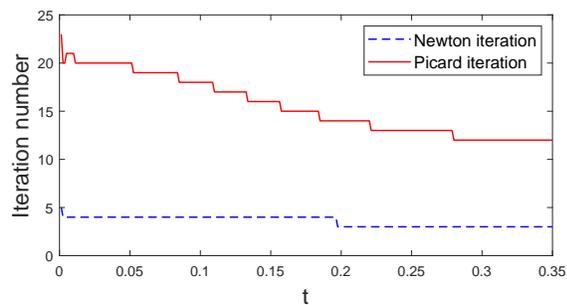}
\caption{A comparison between the number of iterations used in each time step by the Newton's method in \eqref{eq:newton3d} and the Picard iteration in \eqref{eqn:picard3d}, where $h=2^{-3}$ and $\tau=1.25\times 10^{-3}$.}
\label{fig:3dNP}
\end{figure}

\begin{figure}[!htp]
\centering
\begin{minipage}[b]{1\linewidth}
\centering
\includegraphics[width=0.45\textwidth]{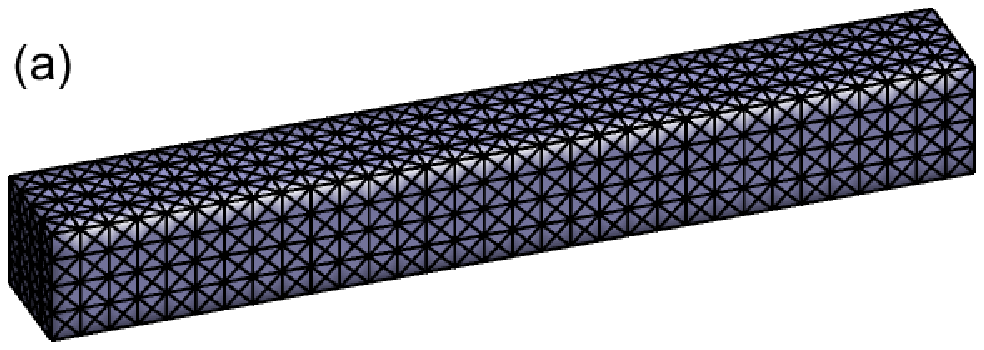}
\hspace{0.8cm}
\includegraphics[width=0.45\textwidth]{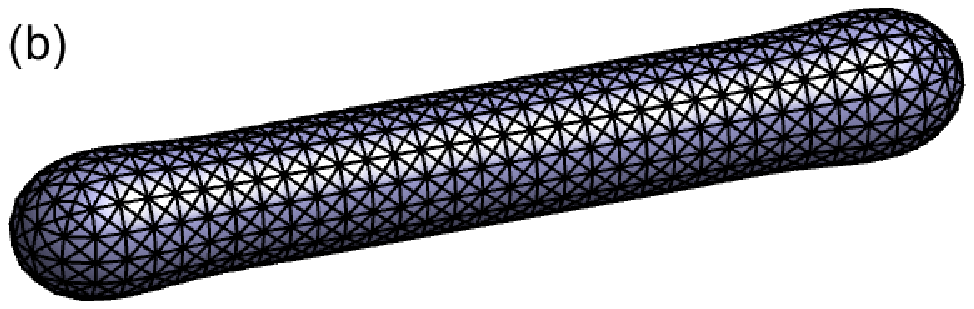}
\end{minipage}
\begin{minipage}[b]{1\linewidth}
\centering
\includegraphics[width=0.45\textwidth]{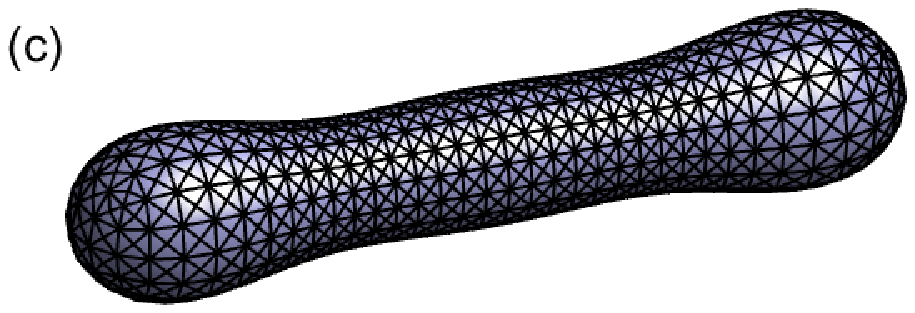}
\hspace{0.8cm}
\includegraphics[width=0.45\textwidth]{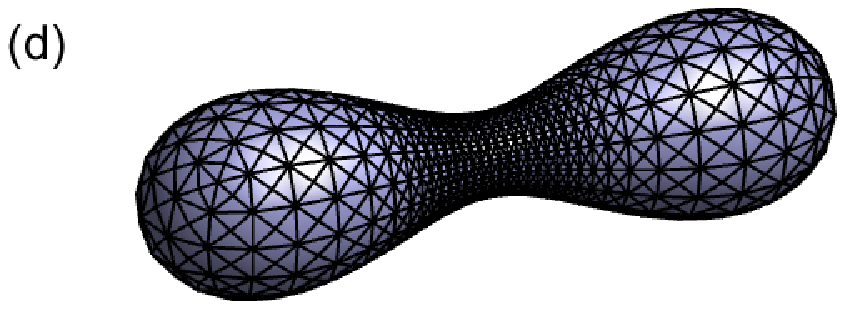}
\end{minipage}
\begin{minipage}[b]{1\linewidth}
\centering
\includegraphics[width=0.45\textwidth]{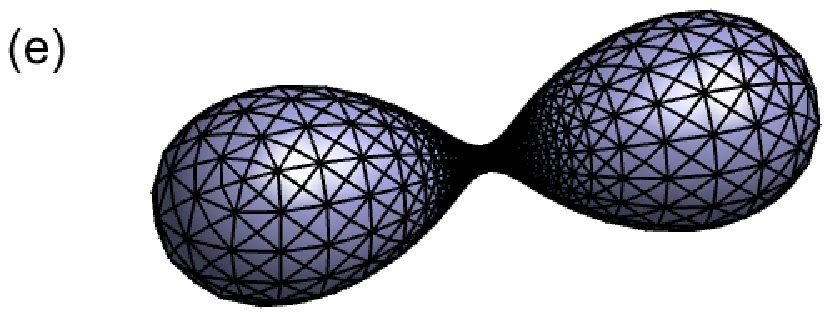}
\hspace{0.8cm}
\includegraphics[width=0.45\textwidth]{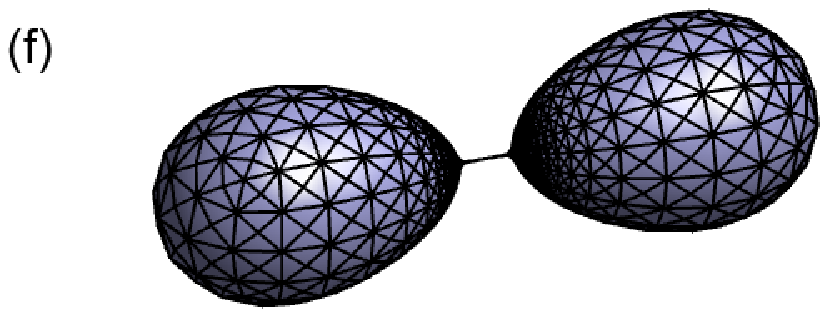}
\end{minipage}
\caption{Several snapshots in the evolution of an initial $(8,~1,~1)$ cuboid until its pinch-off. (a) $t=0$; (b) $t=0.01$; (c) $t=0.1$; (d) $t=0.3$; (e) $t=0.365$; (f) $t=0.370$.}
\label{fig:prism8}
\end{figure}

\begin{figure}[!htp]
\centering
\includegraphics[width=0.95\textwidth]{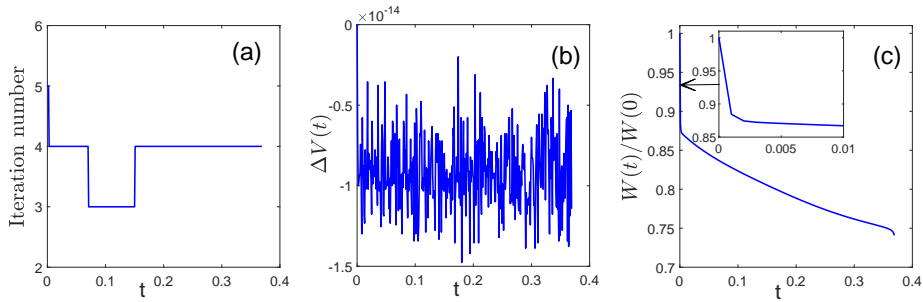}
\caption{(a) The iteration number in each time step by using the Newton's method \eqref{eq:newton3d}. (b) The relative volume loss $\Delta V(t)$. (c) The normalized surface area  $W(t)/W(0)$.}
\label{fig:prism8Q}
\end{figure}

We next consider the shape evolution of an initial $(8,1,1)$ cuboid. Due to the presence of sharp corners, the shape evolves very fast at the very beginning stage (see \cite{Bansch05,Barrett08JCP}). Therefore adaptive time steps are usually required for the simulations in order to accurately predict the pinch-off time. In the current example, we discretize the cuboid into $J=2176$ triangles with $K=1090$ vertices, and choose a uniform time step $\tau=10^{-3}$ for the simulation. Fig.~\ref{fig:prism8} depicts the morphological evolution of the cuboid, where we observe the pinch-off event happens at the time $t=0.370$. This shows a high level of consistency with previous result obtained by using adaptive time steps (see \cite{Barrett08JCP}). In Fig.~\ref{fig:prism8Q}, we plot the iteration number used in each time step, the relative volume loss and the normalized surface area versus time. We find in most time steps, only $4$ iterations are required in the Newton's method, thus it is efficient. We also observe the exact conservation of the volume and decrease of the surface area, as expected by Theorem \ref{thm:3dmassconserve} and Theorem \ref{thm:3dEnergystability}.


\begin{figure}[tph]
\centering
\begin{minipage}[b]{0.93\linewidth}
\centering
\includegraphics[width=0.95\textwidth]{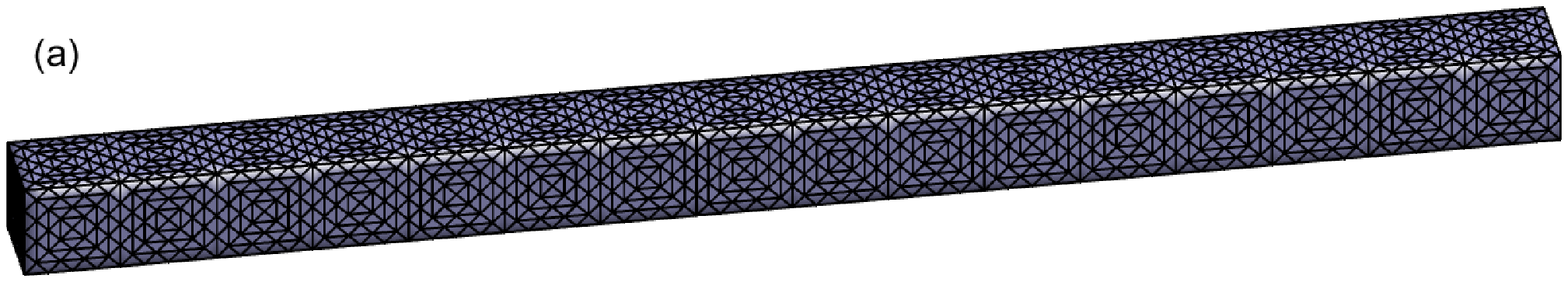}
\end{minipage}
\begin{minipage}[b]{0.92\linewidth}
\centering
\includegraphics[width=.95\textwidth]{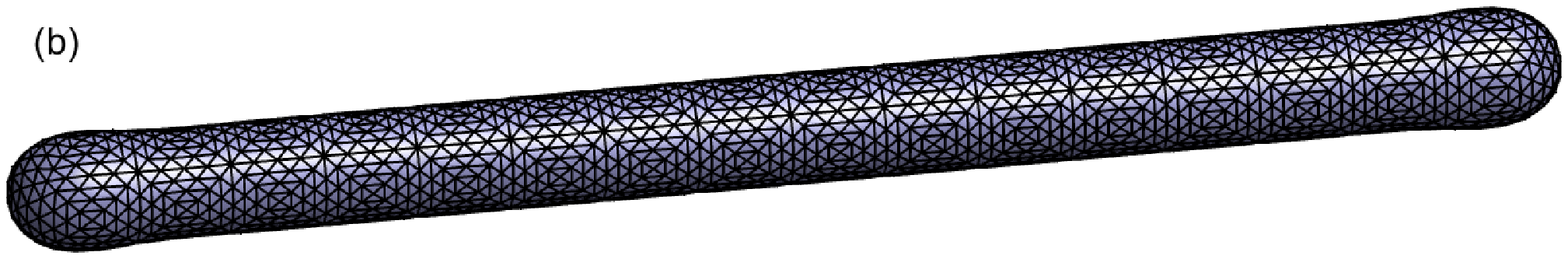}
\end{minipage}
\begin{minipage}[b]{0.92\linewidth}
\centering
\includegraphics[width=0.95\textwidth]{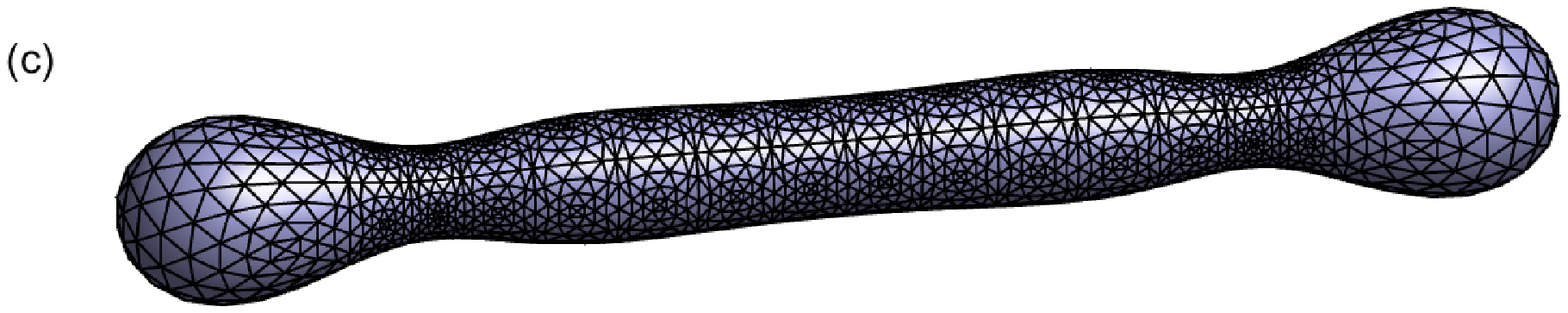}
\end{minipage}
\begin{minipage}[b]{0.92\linewidth}
\centering
\includegraphics[width=0.95\textwidth]{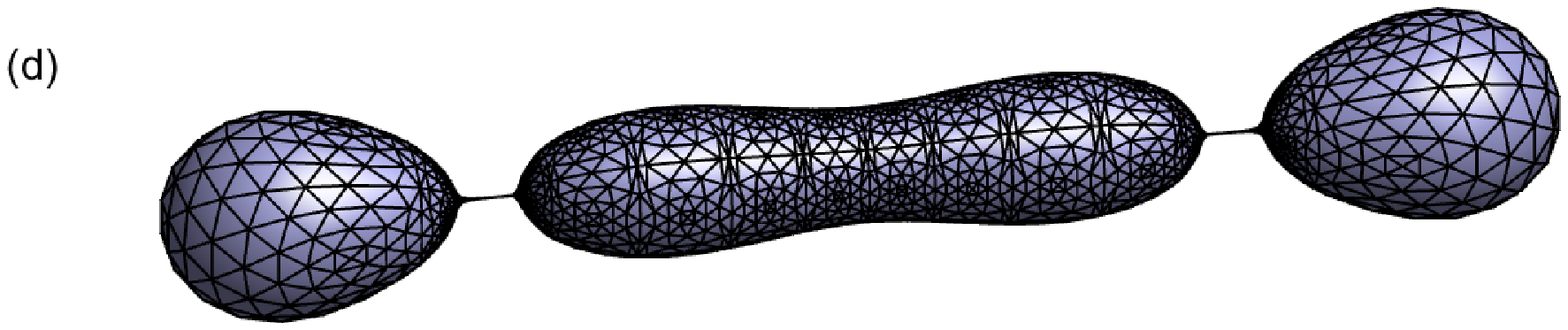}
\end{minipage}
\caption{Several snapshots in the evolution of an initial $(16,~1,~1)$ cuboid until its pinch-off. (a) $t=0$; (b) $t=0.01$; (c) $t=0.4$; (d) $t=0.63$.}
\label{fig:prism16}
\end{figure}

In the last example, we apply our numerical method SP-PFEM to the evolution of a long cuboid of size $(16,1,1)$. We use the computational parameters: $K=2114, J=4224$ and $\tau=10^{-3}$. The numerical results are reported in Fig.~\ref{fig:prism16}, where we observe the formulations of two singularities during the evolution. We note here the pinch-off time is $t=0.630$, which differs slightly from the previous result in \cite{Bansch05} ($t=0.669$).
The discrepancy may be due to the mesh regularization errors or the volume loss for their numerical solutions.

\section{Conclusions}\label{se:con}
We proposed a structure-preserving parametric finite element method (SP-PFEM) for the surface diffusion flow of a 2D curve and 3D surface. The numerical method was based on the discretization of a weak formulation that allows the tangential velocity \cite{Barrett07}. We adopted a ``weakly'' implicit (or almost semi-implicit) discretization in time and piecewise linear elements in space. The key ingredient is that we defined a new vector on average to approximate the unit normal by using the information at the current and next time step. In this sense, the numerical method yielded the good properties of area\slash volume conservation, unconditional stability and good mesh quality. The numerical discretization is ``weakly'' nonlinear in the sense that only one nonlinear term of polynomial form is introduced in each equation of the discrete system, which can be efficiently and accurately solved by the Newton's iterative method.

We assessed the accuracy and convergence of the SP-PFEM {    by numerical tests and it is illustrated} that the order of convergence in spatial discretization can reach about 2 as the mesh size is refined. Various numerical experiments were carried out to verify the good properties of the SP-PFEM. In all, our numerical method provides a reliable and powerful tool for the simulation of surface diffusion flow for 2D curve and 3D surface.

We remark here that the SP-PFEM \eqref{eqn:dis2d} in 2D and \eqref{eqn:3dpfem} in 3D can be straightforwardly extended to the anisotropic surface diffusion flow based on the works in \cite{Barrett20, Zhao19b, Li2020Energy}, the volume-preserving mean curvature flows \cite{Huisken1987volume}, and other curvature driven flows that preserve the volume. Of course, these extensions are required further investigation in terms of preserving the mesh quality, especially the asymptotic equal mesh distribution.


\bibliographystyle{siamplain}
\bibliography{thebib}
\end{document}